\documentclass{amsart}
\usepackage{amsmath, amsfonts, latexsym, array, amssymb, amscd }
\usepackage[all]{xypic}
\usepackage[all]{xy}
\usepackage{graphicx}
\usepackage{color}
\usepackage{enumitem}
\usepackage{epstopdf}
\usepackage{hyperref}

\numberwithin{equation}{section}
\numberwithin{figure}{section}

\newtheorem{thm}{Theorem}[section]

\newtheorem{cor}[thm]{Corollary}
\newtheorem{lem}[thm]{Lemma}

\newtheorem{prop}[thm]{Proposition}

\newtheorem{defn}[thm]{Definition}

\newtheorem{thmx}{Theorem}

\newcommand{\comment}[1]{}

\usepackage{ulem}

\newcommand{\Sing}{\mathrm{Sing}}
\newcommand{\Reg}{\mathrm{Reg}}

\newcommand{\ZZ}{\mathbb{Z}}
\newcommand{\RR}{\mathbb{R}}

\newcommand{\ph}{\varphi}
\newcommand{\PPP}{\mathcal{P}}
\newcommand{\SSS}{\mathcal{S}}
\newcommand{\GGG}{\mathcal{G}}
\newcommand{\RRR}{\mathcal{R}}
\newcommand{\CCC}{\mathcal{C}}

\newcommand{\AAA}{\mathcal{A}}

\newcommand{\MMM}{\mathcal{M}}

\newcommand{\DDD}{\mathcal{D}}
\newcommand{\VVV}{\mathcal{V}}

\newcommand{\sing}{\mathrm{Sing}}
\newcommand{\eps}{\varepsilon}
\newcommand{\vg}{\varphi^u}

\newcommand{\FFF}{\mathcal{F}}
\newcommand{\UUU}{\mathcal{U}}

\DeclareMathOperator{\NE}{NE}

\usepackage{mathrsfs}

\newcommand{\x}{\times}
\newcommand{\R}{\RR}
\newcommand{\N}{\mathbb{N}}

\newcommand{\s}[1]{\mathscr{#1}}

\providecommand{\abs}[1]{\left | #1 \right |}
\providecommand{\norm}[1]{\lVert #1 \rVert}

\newcommand{\es}{\emptyset}

\newcommand{\tPPP}{\tilde {\mathcal{P}}}
\newcommand{\tSSS}{\tilde {\mathcal{S}}}
\newcommand{\tGGG}{\tilde{\mathcal{G}}}

\newcommand{\mukbm}{\mu_{\mathrm{KBM}}}

\begin{document}
\title[Mixing properties of rank 1 geodesic flows] {Equilibrium states for self-products of flows and the mixing properties of rank 1 geodesic flows} 

\author{Benjamin Call and Daniel~J.~Thompson}
\address{B. Call, Department of Mathematics, The Ohio State University, Columbus, OH 43210, \emph{E-mail address:} \tt{call.119@buckeyemail.osu.edu}}
\address{D.~J.~Thompson, Department of Mathematics, The Ohio State University, Columbus, OH 43210, \emph{E-mail address:} \tt{thompson@math.osu.edu}}

\thanks{This work is partially supported by NSF grants DMS-$1461163$ and DMS-$1954463$.}
\date{\today}
\commby{}

\begin{abstract}
Equilibrium states for geodesic flows over closed rank 1 manifolds were studied recently in \cite{BCFT}. For sufficiently regular potentials, it was shown that if the singular set does not carry full pressure then the equilibrium state is unique. The main result of this paper is that these equilibrium states have the Kolmogorov property. In particular, these measures are mixing of all orders and have positive entropy. For the Bowen-Margulis measure, we go further and obtain the Bernoulli property from the Kolmogorov property using classic arguments from Ornstein theory. Our argument for the Kolmogorov property is based on an idea due to Ledrappier. We prove uniqueness of equilibrium states on the product of the system with itself. To carry this out, we develop techniques for uniqueness of equilibrium states which apply in the presence of the 2-dimensional center direction which appears for a product of flows. This is a key technical challenge of this paper.
\end{abstract}
\maketitle
\setcounter{tocdepth}{1}
\section{Introduction}
Let $M= (M^n,g)$ be a closed connected $C^\infty$ Riemannian manifold with non-positive sectional curvature and dimension $n$, and let  $(g_t)_{t\in\mathbb{R}}$ denote the geodesic flow on the unit tangent bundle $T^1M$. The theory of equilibrium states for this setting was developed recently in \cite{BCFT}. Mixing properties are a central topic in ergodic theory. The Bernoulli property is the ultimate mixing property from the measure-theoretic point of view, and the Kolmogorov property is the next strongest mixing property of wide interest. Our main focus is to prove the Kolmogorov property for the class of equilibrium states considered in \cite{BCFT}.  We also establish the Bernoulli property for the measure of maximal entropy.  
Studying the Kolmogorov and Bernoulli properties, and conditions under which $K$ implies Bernoulli, is an active area in smooth ergodic theory, with recent references including \cite{LLS, PTV,  KHV, gP19}.

We set up some preliminaries in order to state our results. The \emph{rank} of a vector $v\in T^1M$ is the dimension of the space of parallel Jacobi vector fields for the geodesic through $v$. The rank is at least $1$ because there is always a parallel Jacobi field corresponding to the flow direction. The \emph{regular set}, denoted $\Reg$, is the set of $v\in T^1M$ with rank $1$. The \emph{singular set}, denoted $\mathrm{Sing}$, is the set of vectors whose rank is larger than 1. We say that  the manifold $M$ is rank $1$ if $\Reg \neq \emptyset$. This is the typical situation in non-positive curvature: if $M$ is irreducible and every $v$ is higher rank, then $M$ is locally symmetric by the rank rigidity theorem \cite{wB95, wB85, BS85}. We assume that $M$ has rank $1$.

We consider equilibrium states for  H\"older continuous potentials or scalar multiples of the \emph{geometric potential} $\vg$. 
We recall that the unstable bundle $v \to E^u_v$ is a continuous invariant subbundle of $TT^1M$,  and that the potential $\vg(v)$ measures infinitesimal expansion in $E^u_v$.  The family of potentials $q\vg$, where $q\in \RR$, are of particular interest in the theory.

For a continuous potential $\ph: T^1M \to \RR$, we let $P(\ph)$ denote the topological pressure with respect to the geodesic flow. We let $P(\Sing, \ph)$ denote the topological pressure of the potential $\ph|_{\Sing}$ with respect to the geodesic flow restricted to the singular set (setting $P(\Sing, \ph) = - \infty$ if $\Sing = \emptyset$, in which case the flow is Anosov).  It was proved in \cite{BCFT} that  if $\varphi\colon T^1M\to \RR$ is $\varphi=q\vg$ or H\"older continuous, and if the pressure gap  $P(\sing, \varphi)<P(\varphi)$ holds, then there is a unique equilibrium state and it is fully supported.  In this paper, we go further in describing the properties of the unique equilibrium states thus obtained, and we prove the following. 
\begin{thmx} \label{thmA}
Let $(g_t)$ be the geodesic flow over a closed rank 1 manifold $M$ and let $\varphi\colon T^1M\to \RR$ be $\varphi=q\vg$ or be H\"older continuous. If $P(\sing, \varphi)<P(\varphi)$, then the unique equilibrium state $\mu_\varphi$ has the Kolmogorov property.
\end{thmx}
In particular,  $\mu_\varphi$ is mixing of all orders, has countable Lebesgue spectrum, and has positive entropy.  We remark that if the pressure gap $P(\sing, \varphi)<P(\varphi)$ fails, then there must exist equilibrium states supported on $\Sing$. Therefore, we can reformulate our result as saying that if $\varphi\colon T^1M\to \RR$ is $\varphi=q\vg$ or H\"older continuous, then a fully supported unique equilibrium state for $\ph$ must have the $K$-property.

When $M$ is a surface, it is already known that the equilibrium state is Bernoulli, and thus $K$, by applying Lima-Ledrappier-Sarig \cite{BCFT, LLS}. Their approach relies on the countable state symbolic dynamics for $3$-dimensional flows established by Lima and Sarig \cite{LS}. For higher dimensional flows, the $K$-property (and even mixing in the case that $\ph$ is not constant) is a new result when $\Sing \neq \emptyset$. We denote the measure of maximal entropy $\mukbm$ after Knieper, Bowen, and Margulis. Babillot proved mixing for $\mukbm$ using product structure of the measure provided by Knieper's construction \cite{mB02}. To the best of our knowledge, stronger mixing properties for $\mukbm$ have not previously been described in the literature when $n \geq 3$.

Our argument for the $K$-property is to follow a remarkable strategy of Ledrappier \cite[Proposition 1.4]{L}, which gives a criterion for the $K$-property in terms of thermodynamic formalism. Consider an asymptotically $h$-expansive topological dynamical system $(X, f)$ and a continuous potential $\varphi$ on $X$, and define a potential $\Phi$ on $X\x X$ by $\Phi(x,y) =\ph(x)+ \ph(y)$. Ledrappier showed that if the product system $(X \times X, f \times f)$ has a unique equilibrium state for $\Phi$, then the equilibrium state for $\ph$ on $(X, f)$ has the $K$-property. We apply a continuous-time analogue of Ledrappier's result to the geodesic flow on a closed rank $1$ manifold. This reduces our analysis to the question of uniqueness of equilibrium states for the system given by the product of a geodesic flow with itself. 

We extend the machinery for uniqueness of equilibrium states developed by Climenhaga and the second named author in \cite{CT} to the class of products of flows. A key idea in \cite{CT} is to find a \emph{decomposition} of the space of orbit segments. This means that any finite-length orbit segment is assigned a `good' core by removing a `bad' segment from the start and from the end. We require that good orbit segments have the specification property and the Bowen property, while the collection of bad orbit segments carries less pressure than the whole system. Uniqueness of equilibrium states for rank 1 geodesic flow was established in \cite{BCFT} by exhibiting this kind of decomposition. We would like to find a decomposition for the product flow, but in general decompositions do not behave well under products. If a collection of orbit segments $\GGG$ has good properties, then we can expect that $\GGG \x \GGG$ does too. However, we need $\GGG \x \GGG$ to arise in a decomposition for $(X \x X, \FFF \x \FFF)$. In general, this does not look at all promising. The issue can be seen clearly if one considers the decomposition for an $S$-gap shift given in \cite{CT12}.

A main idea in our analysis is to formalize the notion of $\lambda$-decompositions for the space of orbit segments $X \x [0, \infty)$. This framework is motivated by, and generalizes, the decomposition for geodesic flow in non-positive curvature which was used in \cite{BCFT}. Unlike a general decomposition, $\lambda$-decompositions induce a natural decomposition on the self-product system. We use this to show that the product system $(T^1M \x T^1M, (g_t \x g_t))$ has a decomposition which satisfies the pressure gap using ideas which extend those in \cite{BCFT}. 
Other examples which are included in our definition of $\lambda$-decompositions include those used in \cite{CFT, CFT2} to study equilibrium states for DA systems, and the decompositions used to study geodesic flow on surfaces with no focal points in \cite{CKP}. 

The other key idea required to apply the machinery of \cite{CT} is to show that the pressure of obstructions to expansivity is smaller than that of the whole space. For a product of flows $(X \x X, \FFF \x \FFF)$, this is never the case. Consider the bi-infinite Bowen balls $\Gamma_{\eps}(x_1 ,x_2)$ for the product flow, defined at \eqref{Gammaeps}. The best expansivity property one can expect is that $\Gamma_\eps(x_1,x_2)$ is contained in the $2$-dimensional set $\{ (f_{r_1}x_1, f_{r_2}x_2) \mid r_1, r_2 \in [-s,s]\}$ for some $s>0$. 
Thus, even if the flow $(X, \FFF)$ is expansive, every non-isolated point in the product flow is non-expansive. We address this by  controlling the `product non-expansive set':
 $$\operatorname{NE^\x}(\eps) := \{(x_1,x_2)\in X\x X\mid \Gamma_\eps(x_1,x_2)\not\subset f_{[-s,s]}(x_1)\x f_{[-s,s]}(x_2)\text{ for any }s > 0\}.$$
We say a $(\FFF\x\FFF)$-invariant measure $\nu$ is \emph{product expansive} if $\nu(\NE^\x(\eps))=0$ for small $\eps>0$. 

We outline some strategy for the uniqueness proof. Our presentation naturally focuses on the case  $\Sing \neq \emptyset$. Of course, our approach also applies in the Anosov case, $\Sing = \emptyset$, which is formally covered by this paper modulo some different (simpler) details to replace the pressure estimates in \S \ref{pressureest}. Let $\mu$ be the unique equilibrium state for $\ph$ for a flow $(X, \FFF)$ provided by an application of the machinery of \cite{CT}. We  show that the measure $\mu\x\mu$ is an equilibrium state for $\Phi$ and that it is product expansive. Furthermore, the weak Gibbs property for the equilibrium state $\mu$ lifts to $\mu \x \mu$. Our task is to show that these ingredients are enough to run the proof that there are no equilibrium states mutually singular to $\mu \x \mu$, and then to show that $\mu \x \mu$ is ergodic.

We encounter technical difficulties related to approximating sets with elements of `adapted partitions' for $(t, \eps)$-separated sets. For a product expansive measure, we show that partition elements can approximate sets invariant under the $\RR^2$-action (i.e. for $f_s \x f_t$ for every $s,t$). However, we are not able to approximate sets invariant under only the diagonal action $(f_s \x f_s)$. This is what would be needed to adapt the ergodicity proof given in \cite{CT}, so we need a different approach:  we prove weak mixing for $\mu$ using a spectral argument, and this gives ergodicity of $\mu \x \mu$. The key ingredient is a `light mixing' property for positive measure sets that have been flowed out by a uniform constant $\beta>0$.

The arguments given in this paper are not specific to geodesic flow. 
Our central argument gives criteria for uniqueness of equilibrium states designed to be applicable for systems which are self-products of flows. Theorem \ref{general} gives the general abstract statement provided by the arguments in this paper, and we expect that it will be broadly applicable beyond the current setting. While our focus is on systems with non-uniform structure, we note that the uniqueness result is new even in the case of a product of uniformly hyperbolic flows, since the expansivity issues we have to deal with are already present in that setting. In \S \ref{s.general}, we discuss remaining room for improvement in the hypotheses of Theorem \ref{general}. These generalizations are not pursued here since it would distract from the main ideas necessary for our analysis. We note that some further applications of our approach in both continuous and discrete-time have been explored in the preprint \cite{bC20} by the first-named author, which is a sequel to this paper.

  We now discuss the Bernoulli property for the measure of maximal entropy $\mukbm$. We use the product structure of the measure provided by Knieper's construction (which is not currently known to extend to equilibrium states), and follow the classic strategy of Ornstein theory to move up the mixing hierarchy. This strategy was carried out for the geodesic flow on constant negative curvature surfaces in \cite{OW73}, and in \cite{yP77, mR74, LLS, CH96, OW98, PTV}. In particular, Pesin showed that the Liouville measure restricted to the regular set is Bernoulli in the current setting. We rigorously extract the statement that `$K$ implies Bernoulli' for $\mukbm$ from Chernov and Haskell's paper \cite{CH96}, whose results are stated for a suspension flow over a non-uniformly hyperbolic map with a smooth measure.  We conclude the following.
\begin{thmx} \label{Bernoulli}
Let $(g_t)$ be the geodesic flow over a closed rank 1 manifold $M$. The unique measure of maximal entropy $\mukbm$ is Bernoulli.
\end{thmx}

We remark that Theorem \ref{Bernoulli} may be anticipated by some experts in this area in light of Babillot's mixing result, because classical arguments using product structure are likely to show that  `mixing implies $K$' for $\mukbm$. With this in hand, Babillot's mixing result would bootstrap all the way to $K$, and then to Bernoulli by the argument presented here. However, neither the proof details nor a precise statement of the needed hypotheses for this approach to the $K$-property have been written for measures with product structure, particularly for flows.  We emphasize that our approach to the $K$-property does not use product structure of the measure.
Furthermore, our proof gives an alternative approach to Babillot's mixing result. We discuss the status of approaches to the $K$-property based  purely on product structure in more detail in \S \ref{pp}.

We note that after the preprint version of this paper appeared, Theorem \ref{Bernoulli} was generalized to equilibrium measures in a recent preprint by Araujo, Lima, and Poletti \cite{ALP}, by extending the symbolic dynamics construction of \cite{LS} to the higher-dimensional setting. This provides product structure at the symbolic level for all the equilibrium measures considered in Theorem \ref{thmA}, and this can be used to improve the $K$-property to Bernoulli.

The paper is structured as follows. In \S \ref{background}, we give background. In \S \ref{products}, we give some general results about product systems. In \S \ref{prodalmost}, we describe properties of product expansive measures.  In \S \ref{pressureest}, we give our pressure estimate for the product of the geodesic flow. In \S \ref{productproof}, we prove that the product system has a unique equilibrium state, completing our proof of the $K$-property. In \S \ref{s.bernoulli}, we describe how to obtain the Bernoulli property for $\mukbm$.
\vspace{10pt}

\section{Background} \label{background}

\subsection{Setting}
We write $(X,d)$ for a compact metric space and $\FFF = (f_t)_{t\in\R}$ a continuous flow on $X$. We write $\MMM(X, \FFF)$ for the space of $\FFF$-invariant Borel probability measures on $X$.  We often consider the metric $d_t(x,y) = \max \{ d(f_sx, f_sy) : s \in[0,t]\}$, and consider metric balls in the $d_t$ metrics, that is the \emph{Bowen balls}
\[
B_t(x, \eps) = \{y : d_t(x, y) < \eps\}.
\]
On occasion, we work with \emph{two-sided Bowen balls}, which we define as
	$$B_{[-t,t]}(x,\eps) = \{y\mid d(f_sx,f_sy) < \eps, s\in [-t,t]\}.$$
We will also consider the product space $X\x X$, which we equip with the metric
	$$\tilde{d}((x,y),(w,z)) = \max\{d(x,w),d(y,z)\}.$$
	
In this metric, it is easy to check that
$B_t((x,y),\eps) = B_t(x,\eps)\x B_t(y,\eps).$ As a notation convention, we write $(f_t)$ when we are considering a general continuous flow, and $(g_t)$ when we are considering geodesic flow.

\subsection{Geodesic flow in non-positive curvature} \label{geoflow}
We collect the necessary definitions to state our results. We refer to \cite{BCFT} for more details, and to \cite{wB95,pE99} for general reference.

Let $M$ be a compact, connected, boundaryless smooth manifold with a smooth Riemannian metric $g$, with non-positive sectional curvatures at every point. For each $v$ in the unit tangent bundle $T^1M$ there is a unique constant speed
geodesic denoted $\gamma_v$ such that $\dot{\gamma}_v(0)=v$.  The \emph{geodesic flow} $(g_t)_{t\in\mathbb{R}}$ acts on $T^1M$  by $g_t(v)=(\dot\gamma_v)(t)$.  We equip $T^1M$ with the Manning-Knieper distance function \cite{aM79, knieper98}:
\begin{equation}\label{eqn:dK}
d(v,w) = \max \{d_M(\gamma_v(t), \gamma_w(t)) \mid t \in [0,1] \},
\end{equation}
where $d_M$ is the distance function on $M$ induced by the Riemannian metric.  

Given $v\in T^1M$, stable and unstable horospheres $H^s_v$ and $H^u_v$ can be defined using a standard geometric construction in the universal cover. For $H^s_v$, we consider the set of points in the universal cover $\tilde M$ at distance $r$ from $g_rv$, that is
\[
S^r(v, +) = \{ x \in \tilde M : d_{\tilde M}(x, g_rv) = r\},
\]
and we take the limit of $S^r(v, +)$ as $ r \to \infty$. This defines a hypersurface which contains the point $\pi v$. The stable horosphere $H^s_v$ is the projection to $M$ (from $\tilde M$) of this hypersurface \cite[Proposition 2.6]{pE73}. The stable manifold $W^s_v$ is the normal unit vector field to $H^s_v$ on the same side as $v$. For $H^u_v$, we consider the set of points in $\tilde M$ at distance $r$ from $g_{-r}v$, that is
\[
S^r(v, -) = \{ x \in \tilde M : d_{\tilde M}(x, g_{-r}v) = r\},
\]
and we take the limit of $S^r(v, -)$ as $ r \to \infty$. The projection to $M$ of this hypersurface is the stable horosphere $H^u_v$. The unstable manifold $W^u_v$ is the normal unit vector field to $H^u_v$ on the same side as $v$. The horospheres are $C^2$ manifolds, and we can define  the stable and unstable subspaces $E^s_v, E^u_v \subset T_vT^1M$ to be the tangent spaces of $W^s_v, W^u_v$ respectively.  The bundles $E^s, E^u$, which are both globally defined in this way, are respectively called the stable and unstable bundles. The bundles $E^s, E^u$ are invariant, and depend continuously on $v$, see \cite{pE99, GW99}.  
We can define the \emph{geometric potential}, to be
\[
\vg(v)=-\lim_{t\to 0} \frac{1}{t}\log \det(dg_t|_{E^u_v}).
\]
The geometric potential is globally defined and continuous.

We define the \emph{singular set} $\Sing$ to be the set of $v$ such that $E^s_v$ and $E^u_v$ intersect non-trivially. The set $\Sing$ is closed and invariant. We define the \emph{regular set} $\Reg$ to be the complement of $\Sing$ in $T^1M$. An alternative construction of $E^s, E^u$ is given infinitesimally using stable and unstable Jacobi fields. These bundles can be shown to be integrable, and $W^s, W^u$ is characterized as the foliation obtained by integrating these bundles. With this approach, $\Sing$ is defined as the set of $v \in T^1M$ so that the geodesic determined by $v$ has a parallel orthogonal Jacobi field. This can be seen to be equivalent to the definition of $\Sing$ given above. The Jacobi field formalism is used extensively in \cite{BCFT}, and we refer there for full definitions.

We define a function $\lambda\colon T^1M\to [0,\infty)$ as follows.  Let $H^s, H^u$ be the stable and unstable horospheres for $v$. Let $\UUU^s_v \colon T_{\pi v} H^s \to T_{\pi v} H^s$ be the symmetric linear operator defined by $\UUU(v)=\nabla_vN$, where $N$ is the field of unit vectors normal to $H$ on the same side as $v$.  This determines the second fundamental form of the stable horosphere $H^s$. We define $\UUU^u_v \colon T_{\pi v} H^u \to T_{\pi v} H^u$ analogously. Then $\UUU_v^u$ and $\UUU_v^s$ depend continuously on $v$, $\UUU^u$ is positive semidefinite, $\UUU^s$ is negative semidefinite, and $\UUU^u_{-v}=-\UUU^s_v$.  
\begin{defn}
For $v \in T^1M$, let $\lambda^u(v)$ be the minimum eigenvalue of $\UUU^u_v$ and let $\lambda^s(v) = \lambda^u(-v)$. Let $\lambda(v) = \min ( \lambda^u(v), \lambda^s(v))$.
\end{defn}
The functions $\lambda^u$, $\lambda^s$, and $\lambda$ are continuous since the map $v\mapsto \UUU^{u,s}_v$ is continuous.
By positive (negative) semidefiniteness of $\UUU^{u,s}$, we have $\lambda^{u,s} \geq 0$.  When $M$ is a surface, the quantities $\lambda^{u,s}(v)$ are just the curvatures at $\pi v$ of the stable and unstable horocycles. 

If $v \in \Sing$, then $\lambda(v)=0$ due to the presence of a parallel orthogonal Jacobi field. The set $\{v \in \Reg: \lambda(v)=0\}$ may be non-empty, but it has zero measure for any invariant measure \cite[Corollary 3.6]{BCFT}. If $\lambda(v) \geq \eta >0$, then we have various uniform estimates at the point $v$, for example on the growth of Jacobi fields at $v$ \cite[Lemma 2.11]{BCFT} and the angle between $E^u_v$ and $E^s_v$ \cite[\S3.3]{BCFT}. Thus, the function $\lambda$ serves as a useful  `measure of hyperbolicity'.

\subsection{The $K$-property}
We give a brief survey of the Kolmogorov property. For a more extensive survey, we refer to Chapter 10.8 in \cite{CFSS}. 
The $K$-property is a mixing property, stronger than mixing of all orders and weaker than Bernoulli. The original definition of the $K$-property for a discrete-time system is as follows.

\begin{defn}\label{K-definition}
	Let $f : X\to X$ be an invertible measure-preserving transformation. Then we say that $f$ is 
\emph{Kolmogorov}, or that the system has the \emph{$K$-property}, if there is a sub-$\sigma$-algebra $\mathscr{K}$ of $\mathscr{B}$ which satisfies $f\mathscr{K}\supset\mathscr{K}$, $\bigvee_{i=0}^\infty f^{i}\mathscr{K} = \mathscr{B}$, and $\bigcap_{i=0}^\infty f^{-i}\mathscr{K} = \{\es, X\}.$
\end{defn}
A system $(X,f,\mu)$ has the $K$-property if and only if it has completely positive entropy, i.e. $h_\mu(f, \xi)>0$ for any partition $\xi \neq \{ \emptyset, X \}$ mod $0$ measure sets. This immediately implies that if $(X,f,\mu)$ has the $K$-property, then $h_\mu(f) > 0$.

There is another equivalent definition of the $K$-property, called $K$-mixing. We say $(X,\s{B},f,\mu)$ is \emph{$K$-mixing} if for any sets $A_0,A_1,\ldots, A_r\in \s{B}$ for $r \geq 0$, we have
	$$\lim\limits_{n\to\infty}\sup_{B\in \s{C}_n^\infty(A_1,\cdots, A_r)}\abs{\mu(A_0\cap B) - \mu(A_0)\mu(B)} = 0,$$
	where $\s{C}_n^\infty(A_1,\cdots, A_r)$ is the minimal $\sigma$-algebra generated by $f^kA_i$ for $1\leq i\leq r$ and $k\geq n$. A system $(X,f,\mu)$ has the $K$-property if and only if it is $K$-mixing.  As a corollary, we see that the $K$-property implies mixing of all orders. Thus, the $K$-property is interpreted as a strong mixing property.  The Bernoulli property, which is the strongest property in the hierarchy of mixing properties, implies the $K$-property \cite[Theorem 4.30]{Wa}. We now define the $K$-property for a flow.
\begin{defn}
A measure-preserving flow  $(X, \FFF, \mu)$ has the $K$-property if for every $t \neq 0$, the discrete-time invertible measure preserving system $(X, f_t, \mu)$ has the $K$-property.
\end{defn}
Rudolph proved in \cite{R} that this definition is equivalent to the natural continuous-time analogue of Definition \ref{K-definition}. It follows from work of Gurevi\v{c} \cite{Gur67} that a flow is $K$ in the sense above if we can check that a single time-$t$ map is $K$.  We give a short self-contained proof, since we will use this criterion in this paper.

\begin{prop}\label{TimeTK}
Let $\FFF=(f_t)$ be a continuous flow and $\mu$ be an $\FFF$-invariant measure. If there exists  $t\in\R$ such that $(X,f_t,\mu)$ is a $K$-system, then $(X,\FFF,\mu)$ is a $K$-flow.
\end{prop}
\begin{proof}
We prove the contrapositive. Suppose that $(X,\FFF,\mu)$ is not a $K$-flow. Then there exists $t_0\in\R\setminus\{0\}$ such that $(X,f_{t_0},\mu)$ is not a $K$-system, and so has a non-trivial Pinsker algebra $\pi(f_{t_0})$. Now, for all $t\in\R$ and $\s{A}\subset \pi(f_{t_0})$, we have that
$$h_\mu(f_{t_0},f_t\s{A}) = h_\mu(f_{t_0},\s{A}) = 0$$
and so $f_t\s{A}\subset \pi(f_{t_0})$. Therefore, $\pi(f_{t_0})$ is $\FFF$-invariant. Consequently, considering the system $(X,\pi(f_{t_0}),\mu,\FFF)$, for all $t\neq 0$,
$$h_\mu(f_t|_{\pi(f_{t_0})}) =  \abs{\frac{t}{t_0}}h_\mu(f_{t_0}|_{\pi(f_{t_0})}) = 0.$$
Thus, $\pi(f_t)$ contains $\pi(f_{t_0})$ which is nontrivial, and so we have shown that $(X,f_t,\mu)$ is not a $K$-system. This completes the proof.
\end{proof}

From this definition, it is easy to see that the properties of mixing of all orders and positive entropy hold for $K$-flows as well.

\subsection{Ledrappier's criterion}
The major tool we use for proving the $K$-property is the following theorem from \cite{L}.

\begin{thm}[Ledrappier]\label{Ledr}
	Let $(X,f)$ be an asymptotically $h$-expansive system, and let $\ph$ be a continuous function on $X$. Let $(X\x X,f\x f)$ be the product of two copies of $(X,f)$ and $\Phi(x_1,x_2) = \ph(x_1) + \ph(x_2)$. If $\Phi$ has a unique equilibrium measure in $\MMM(X\x X,f\x f)$, then the unique equilibrium measure for $\ph$ in $\MMM(X,f)$ has the Kolmogorov property.
\end{thm}

In \cite{L}, the result is stated with a hypothesis called weak expansivity in place of asymptotic $h$-expansivity. However, in \cite{Led78} he demonstrates that this weak expansivity property is equivalent to the now standard definition of asymptotic $h$-expansivity. See also the book \cite{FH} for a contemporary account.

Ledrappier observed that in discrete-time, his theorem applies under Bowen's hypotheses of the specification property, expansivity, and the Bowen regularity property, since all of these properties lift to the product system. Note that the continuous-time analogue of Bowen's hypotheses \cite{eF77} do not lift to the product system since a product of expansive flows is not expansive. Thus, there is a new difficulty that must be overcome to apply this approach for flows, even in the uniform setting.  

We give a short proof that Ledrappier's result generalizes to flows. The involved part of our analysis will be to apply it using suitable weak non-uniform versions of Bowen's hypotheses. First, we give a useful lemma.

\begin{lem}\label{ProductES}
	Let $\mu$ be an equilibrium state for $(X,\FFF,\ph)$. Then $\mu\x\mu$ is an equilibrium state for $(X\x X,\FFF\x \FFF,\Phi)$.
\end{lem}

\begin{proof}
	Observe that $h_{\mu\x\mu}(f_1\x f_1) = 2 h_\mu(f_1)$
	and
	$\int \Phi\,d(\mu\x\mu) = 2\int\ph\,d\mu. $
	Therefore,
	$h_{\mu\x\mu}(f_1\x f_1) + \int\Phi\,d(\mu\x\mu) = 2P(X,\FFF,\ph) = P(X\x X,\FFF\x\FFF,\Phi).$ The last equality follows from \cite[Theorem 9.8]{Wa}, or as a special case of Proposition \ref{Addition}.
\end{proof}
The following continuous-time version of Ledrappier's theorem is proved by reducing to the discrete-time case, following a similar strategy to \cite[Theorem 4.4.1]{FH}.
\begin{prop} \label{Kprop}
	Let $(X,\FFF)$ be a continuous flow on a compact metric space such that $f_t$ is asymptotically $h$-expansive for all $t\neq 0$, and let $\ph$ be a continuous function on $X$. Let $(X\x X,\FFF\x\FFF)$ be the product of two copies of $(X,\FFF)$,  i.e. the flow $(f_s \x f_s)_{s \in \RR}$ given by 
\[
(f_s\times f_s)(x,y) = (f_sx, f_sy) \text{ for }s \in \RR.
\] 
Define the potential $\Phi: X\x X \to \RR$ by $\Phi(x,y) = \ph(x) + \ph(y)$. If $\Phi$ has a unique equilibrium state in $\MMM(X\x X,\FFF\x\FFF)$, then the unique equilibrium state for $\ph$ in $\MMM(X,\FFF)$ has the Kolmogorov property.
\end{prop}

\begin{proof}
	Let $\mu$ be the unique equilibrium state for $(X,\mathcal{F},\ph)$, and $\mu\x\mu$ the unique equilibrium state for $(X\x X,\FFF\x\FFF)$. We claim that $\mu\x\mu$ is the unique equilibrium state for $(X\x X,f_1\x f_1,\Phi_1)$ where $\ph_1 = \int_{0}^{1}\ph\circ f_s\,ds$, and $\Phi_1(x,y) = \ph_1(x) + \ph_1(y)$.  Let $\nu$ be an ergodic equilibrium state for $(X\x X,f_1\x f_1,\Phi_1)$, and let $\tilde{\nu} = \int_{0}^{1}(f_s \x f_s)_*\nu\,ds$. We see that $\tilde{\nu}$ is $(\FFF\x\FFF)$-invariant and also
	\begin{align*}
	h_{\tilde{\nu}}(f_1\x f_1) + \int\Phi\,d\tilde{\nu}
	&= h_\nu(f_1\x f_1) + \int\int_{0}^{1}\Phi\circ (f_s\x f_s)\,ds\,d\nu
	\\
	&= h_\nu(f_1\x f_1) + \int \Phi_1 \,d\nu
\comment{	\\
	&= P_\nu(X\x X,f_1\x f_1,\Phi_1)}
	\\
	&= P(X\x X,f_1\x f_1,\Phi_1)
	\geq P(X\x X,\FFF\x\FFF,\Phi).
	\end{align*}
	\comment{The proof of this last inequality is as follows. Observe that $\int \int_{0}^{1}\ph\circ f_s\,ds\x\int_{0}^{1}\ph\circ f_s\,ds\,d\mu\x\mu = \int \ph\x\ph\,d\mu\x\mu$ by flow-invariance of $\mu$.}

	Thus, $\tilde{\nu}$ is an equilibrium state for $(X\x X,\FFF\x\FFF,\Phi)$, and consequently, is equal to $\mu\x\mu$. 	Since $\mu \x \mu$ is ergodic for $\FFF\x\FFF$, $\mu \x \mu$ is also weak mixing. A proof of this can be found by adapting the arguments in \cite[Theorems 1.21, 1.24]{Wa} to continuous time. It follows from $\S 5.8$ of \cite{Rok} that a flow is weak mixing if and only if every time-$t$ map is ergodic, for $t\neq 0$. Hence, $(X\x X,f_1\x f_1,\mu\x\mu)$ is ergodic. 
	
	Let $G$ be the set of $(f_1\x f_1)$-generic points of $\mu\x\mu$. Because $\mu\x\mu$ is flow-invariant, it follows that $(f_s\x f_s)G = G$ for all $s\in\R$. Using this and the fact that $\mu\x\mu = \int_0^1 (f_s\x f_s)_*\nu\,ds$, we see that
	$$1 = \int_{0}^{1}(f_s\x f_s)_*\nu(G)\,ds = \int_{0}^{1}\nu(f_s\x f_sG)\,ds = \int_{0}^{1}\nu(G)\,ds.$$
	Therefore, because $\nu$ is ergodic, $\nu = \mu\x\mu$. Thus, $(X\x X,f_1\x f_1,\Phi_1)$ has a unique equilibrium state. It follows from Theorem \ref{Ledr} that $(X,f_1,\mu)$ has the $K$-property. Thus, by Proposition \ref{TimeTK}, $(X, \FFF ,\mu)$ is a $K$-flow.
\end{proof}

\subsection{Topological pressure and uniqueness of equilibrium states} \label{pressure} Our approach to showing uniqueness of equilibrium states is based on a general theorem in \cite{CT}. We provide the necessary definitions to understand this framework. A subset $\CCC \subset X\x[0, \infty)$ should be thought of as a collection of orbit segments of the flow via the identification
$$(x,t)\in\CCC \leftrightarrow \{ f_sx \mid s\in [0,t)\}.$$

Given $\CCC\subset X\x [0, \infty)$, for all $t\geq 0$, define $\CCC_t = \{x\in X\mid (x,t)\in\CCC\}$. We say that a set $E$ is $(t,\eps)$-separated if given $x,y\in E$, $B_t(x,\eps)$ and $B_t(y,\eps)$ are disjoint. For all $\eps > 0$, define 
\begin{align*}
\Lambda_t(\CCC,\ph,\eps) &= \sup\{\sum_{x\in E} e^{\int_0^t\ph(f_sx)\,ds}\mid E\subset \CCC_t \text{ is }(t,\eps)\text{-separated}\}.
\end{align*}
It suffices to consider $E\subset \CCC_t $ which are $(t,\eps)$-separated set of maximal cardinality in $\CCC_t$, or else we would be able to increase the sum by adding another point to $E$. We say such a set $E$ is \emph{maximizing} for $\Lambda_t(\CCC,\ph,\eps)$ if it achieves the supremum. We define
$P(\CCC,\ph,\eps) = \limsup \frac{1}{t}\log \Lambda_t(\CCC,\ph,\eps)$
and 
$P(\CCC,\ph) = \lim\limits_{\eps\to 0}P(\CCC,\ph,\eps)$. 

If $\CCC$ is of the form $Z\x [0, \infty)$, then we write $\Lambda_t(Z,\ph,\eps)$ instead of $\Lambda_t(\CCC,\ph,\eps)$. In this case, $P(\CCC,\ph)$ is just the upper capacity pressure of the set $Z$, and we can write $P(Z, \ph)$. If $Z=X$, then we recover the standard topological pressure of the potential $\ph$ on the flow $(X, \FFF)$, and we write $P(\ph)$. We note that  maximizing $(t, \eps)$-separated sets for $\Lambda_t(X,\ph,\eps)$ always exist by compactness. The following lemma is a straight-forward exercise.

	\begin{lem}\label{Maximum}
		Given two collections $\CCC,\DDD\subset X\x [0, \infty)$, then
		$$P(\CCC\cup \DDD,\ph) = \max\{P(\CCC,\ph), P(\DDD,\ph)\}.$$
	\end{lem}
For an invariant measure $\mu$, we write $P_\mu(\ph)$ for the \emph{free energy}
\[
P_\mu(\ph) = h_\mu(f) + \int\ph\,d\mu.
\]
We make the following definitions.

	\begin{defn}\label{empirical}
	For  $(x,t)\in X\x [0,\infty)$, define the \emph{empirical measure} $\s{E}_{(x,t)}$ by 
	$$\int \psi\,d\s{E}_{(x,t)} = \frac{1}{t}\int_{0}^{t}\psi(f_sx)\,ds$$
	for $\psi\in C(X)$. For a collection, $\CCC\subset X\x [0, \infty)$, we define $\MMM_t(\CCC)$ to be the set of convex combinations of empirical measures $\s{E}_{(x,t)}$ for points $x \in \CCC_t$, i.e. $$\MMM_t(\CCC) = \{ \sum_{i=1}^k a_i \s{E}_{x_i, t} : a_i \geq 0, \sum a_i=1, (x_i, t) \in \CCC \}.$$
	We define $\MMM(\CCC)$ to be the set of accumulation points of measures in $\MMM_t(\CCC)$.  That is, $\MMM(\CCC) = \{  \mu = \lim_{k \to \infty} \mu_{k} : \mu_k \in \MMM_{t_k}(\CCC), t_k \to \infty\}$.
	\end{defn}
	We recall a pressure estimate, which is proved as \cite[Proposition 5.1]{BCFT}.
	\begin{prop} \label{bcftest}
		Let $\CCC\subset X\x [0, \infty)$. Then $P(\CCC,\ph)\leq \sup_{\mu\in \MMM(\CCC)}P_\mu(\ph)$.
	\end{prop}
Following \cite{CT}, we define a decomposition of the space of orbit segments.

\begin{defn}\label{def.decomp} 
A \emph{decomposition for $X\times [0, \infty)$} consists of three collections $\PPP, \GGG, \SSS\subset X\times [0, \infty)$ for which there exist three functions $p, g, s\colon X\times [0, \infty)\rightarrow [0, \infty)$ such that for every $(x,t)\in X\times [0, \infty)$, the values $p=p(x,t)$, $g=g(x,t)$, and $s=s(x,t)$ satisfy $t=p+g+s$, and 
\[
(x,p)\in \PPP,\quad
(f_p(x), g)\in \GGG,\quad
(f_{p+g}(x), s)\in \SSS.
\]
 For any $M\in[0,\infty)$, define $\GGG^M = \{(x,t)\mid p(x,t)\leq M, s(x,t)\leq M\}$.
\end{defn}
The idea is that $\GGG$ should have `nice' properties, and that $\PPP, \SSS$ are smaller than the whole space in terms of topological pressure. One of these `nice' properties is the specification property. A fairly strong version of this, which we verify for certain orbit segments in \cite[Theorem 4.1]{BCFT}, is given as follows.
\begin{defn}
A collection of orbit segments $\mathcal{C}\subset X\times [0, \infty)$ has \emph{specification at scale $\rho>0$} if  there exists $\tau=\tau(\rho)$ such that for every $(x_1, t_1)$, $\dots, (x_N, t_N)\in \mathcal{C}$  and \emph{every} collection of times $\tau_1, \ldots, \tau_{N-1}$ with $\tau_i \geq \tau$ for all $i$, there exists a point $y\in X$ such that for $s_0=\tau_0=0$ and $s_j=\sum_{i=1}^j t_i + \sum_{i=0}^{j-1} \tau_i$, we have
\[
f_{s_{j-1} + \tau_{j-1}}(y)\in B_{t_j}(x_j, \rho)
\]
for every $j\in \{ 1,\dots, N\}$.  A collection $\mathcal{C}\subset X\times [0, \infty)$ has \emph{specification} if it has specification at all scales.
\end{defn}
The other `nice' property we ask for is the Bowen property for a collection of orbit segments.

	\begin{defn}
		We say that $\ph : X\to\R$ has the Bowen property at scale $\eps > 0$ on $\CCC\subset X\x[0,\infty)$ if
		$$V(\CCC,\ph,\eps) := \sup\left \{\abs{\int_{0}^{t}[\ph(f_sx) - \ph(f_sy)]\,ds}\mid (x,t)\in\CCC, y\in B_t(x,\eps) \right \} < \infty.$$
			\end{defn}

We consider a certain weak expansivity property. Let $\eps > 0$. Then define
$$\Gamma_\eps(x) := \{y\mid d(f_tx,f_ty) < \eps \text{ for all }t\in\R\}.$$
For a flow, following \cite{CT}, we define the set of non-expansive points at scale $\eps$ to be
$$\NE(\eps) := \{x\mid \Gamma_\eps(x)\not\subset f_{[-s,s]}(x) \text{ for all }s\in\R\}.$$
We say a $\FFF$-invariant measure $\nu$ is \textit{almost expansive} at scale $\eps$ if $\nu(\NE(\eps)) = 0$. We define the pressure of obstructions to expansivity at scale $\eps$,
$$P_{\text{exp}}^\perp(\ph,\eps) = \sup_{\nu}\{h_\nu(f_1) + \int \ph\,d\nu\mid \nu(\NE(\eps)) > 0\},$$
by taking a supremum over all non-almost expansive measures. Then define
$$P_{\text{exp}}^\perp(\ph) = \lim\limits_{\eps\to 0}P_{\text{exp}}^\perp(\ph,\eps).$$
Given a collection  $\CCC$, we define a related `discretized' collection by
\[
[\CCC] := \{(x,n) \in X\times \mathbb N \mid (f_{-s}x, n+s+t) \in \CCC \text{ for some }s,t\in [0,1]\}.
\]
We can now state the abstract theorem for uniqueness of equilibrium states proved in \cite{CT}.
\begin{thm}[Climenhaga-Thompson] \label{CT}
	Let $(X,\FFF)$ be a continuous flow on a compact metric space, and $\ph : X\to \R$ a continuous potential. Suppose that $P_{\text{exp}}^\perp(\ph) < P(\ph)$ and $X\x [0, \infty)$ admits a decomposition $(\PPP,\GGG,\SSS)$ with the following properties:
	\begin{enumerate}
		\item $\GGG$ has specification at all scales;
		\item $\ph$ has the Bowen property on $\GGG$;
		\item $P([\PPP]\cup[\SSS],\ph) < P(\ph)$.
	\end{enumerate}
	Then $(X,\FFF,\ph)$ has a unique equilibrium state.
\end{thm}
This is applied in \cite{BCFT} to give the following result.
\begin{thm} [Burns-Climenhaga-Fisher-Thompson]
Let $(g_t)_{t\in\R}$ be the geodesic flow over a closed rank 1 manifold $M$ and let $\varphi\colon T^1M\to \RR$ be $\varphi=q\vg$ or be H\"older continuous. If $P(\sing, \varphi)<P(\varphi)$, then there exists a unique equilibrium state $\mu_\varphi$.
\end{thm}
Our strategy is to adapt the abstract result of Theorem \ref{CT} to obtain uniqueness of equilibrium states for the product system $(T^1M \x T^1M, (g_t \x g_t))$.  This involves finding a suitable decomposition for $(T^1M \x T^1M, (g_t \x g_t))$ which satisfies properties (1), (2), and the pressure gap (3). However, even when this has been achieved, Theorem \ref{CT} will not apply directly because the expansivity condition  $P_{\text{exp}}^\perp(\ph) < P(\ph)$ is never satisfied for a system which is the product of two flows. Adapting the proof of Theorem \ref{CT} to cover the necessary notion of expansivity for a product of two flows is a major technical point of our argument.

	\section{Products of collections of orbit segments} \label{products}
	Given two collections of orbit segments $\CCC, \DDD \subset X\x[0, \infty)$, we define the \emph{product collection} to be
		$$\CCC\x\DDD := \{((x,y),t)\mid (x,t)\in\CCC \text{ and }(y,t)\in\DDD\}.$$
The set 	$\CCC\x\DDD$ is interpreted as a collection of orbit segments for the product flow $(X \times X, \FFF \times \FFF)$ by the identification
\[
((x,y),t) \longleftrightarrow \{ (f_sx, f_sy) \mid s \in [0, t)\}.
\]
In this section, we give general results on lifting results on collections of orbit segments to products of collections of orbit segments. 

\begin{lem} \label{specprod}
		Suppose $\CCC\subset X\x[0, \infty)$ has the specification property for $\FFF$. Then $\CCC\x \CCC$ has specification for $\FFF \times \FFF$.
	\end{lem}
	\begin{proof}
		Let $\rho > 0$, and let $\tau = \tau(\rho)$ be the specification constant on $\CCC$. Now consider $(x_1,y_1,t_1),\ldots,(x_N,y_N,t_N)\in\CCC\x\CCC$ and an arbitrary collection of times $\tau_1,\ldots,\tau_{N-1}$ satisfying $\tau_i\geq \tau$ for all $i$. By the specification property for $\CCC$, there exist $x,y\in X$ such that setting $s_0 = \tau_0 = 0$ and $s_j$ as in the definition above, we have 
		$$f_{s_{j-1} + \tau_{j-1}}(x)\in B_{t_j}(x_j,\rho) \text{ and }f_{s_{j-1} + \tau_{j-1}}(y)\in B_{t_j}(y_j,\rho)$$
		for $j\in \{1,\ldots,N\}$. This implies that 
		$$(f\x f)_{s_{j-1} + \tau_{j-1}}(x,y)\in B_{t_j}(x_j,\rho)\x B_{t_j}(y_j,\rho) = B_{t_j}((x_j,y_j),\rho)$$
		for all $j\in\{1,\ldots,N\}$. Thus, $(x,y)$ is a point fulfilling the specification property.
	\end{proof}
Note that the weak version of specification considered in \cite{CT} in which we only ask for transition times that are bounded above by $\tau$ does not lift to the product.

		\begin{lem} \label{bowenprod}
		Suppose $\ph : X\to\R$ has the Bowen property at scale $\eps >0$ on $\CCC\subset X\x[0, \infty)$. Then $\Phi: X \x X \to \RR $ defined by $\Phi(x, y) = \ph(x)+\ph(y)$ has the Bowen property at scale $\eps$ on $\CCC\x\CCC$.
	\end{lem}
	\begin{proof}
		Let $((x,y),t)\in\CCC\x\CCC$ and let $(w,z)\in B_t((x,y),\eps)$. Observe that $$\abs{\Phi(f_sx,f_sy) - \Phi(f_sw,f_sz)} \leq \abs{\ph(f_sx) - \ph(f_sw)} + \abs{\ph(f_sy) - \ph(f_sz)}.$$
		Since $w\in B_t(x,\eps)$ and $z\in B_t(y,\eps)$, and $\ph$ has the Bowen property, this gives
		\[
		\abs{\int_{0}^{t} [\Phi(f_sx,f_sy) - \Phi(f_sw,f_sz)]\,ds} \leq V(\CCC,\ph,\eps) + V(\CCC,\ph,\eps) < \infty. \qedhere
		\] 
	\end{proof}
\begin{prop}\label{Addition}
	Let $\mathcal{C},\mathcal{D}\subset X\x [0, \infty)$. Let $\ph_1,\ph_2 : X\to\R$ be continuous, and let $\Phi(x,y) = \ph_1(x) + \ph_2(y)$. Then we have
	$$P(\mathcal{C}\x\mathcal{D},\Phi; \FFF \x \FFF) \leq P(\mathcal{C},\ph_1;\FFF) + P(\DDD,\ph_2;\FFF).$$
	Furthermore, if $\CCC = \DDD$ and $\ph_1 = \ph_2$, we get equality.
\end{prop}
\begin{proof}
	
	For the inequality, we need the following characterization of pressure via spanning sets, which is proved in \cite{CFT}:
	$$P(\CCC,\ph;\FFF) = \lim\limits_{\eps\to 0}\limsup_{t\to\infty}\frac{1}{t}\log\Lambda_t^{\text{span}}(\CCC,\ph,\eps;\FFF),$$
	where $\Lambda_t^{\text{span}}(\CCC,\ph,\eps;\FFF)$ is defined similarly to $\Lambda_t(\CCC,\ph,\eps;\FFF)$, replacing $\sup$ with $\inf$ and separating sets with spanning sets. Using this, we follow the proof of the corresponding inequality in \cite[Theorem 9.8(v)]{Wa}. Observe that given two sets $E_1,E_2\subset X$, we have
	\begin{align*}
	&\sum_{(x_1,x_2)\in E_1\x E_2}\exp\int_{0}^{t}\Phi((f\x f)_s(x_1,x_2))\,ds 
	\\
	= &\left(\sum_{x_1\in E_1} \exp\int_{0}^{t}\ph_1(f_sx_1)\,ds\right)\left(\sum_{x_2\in E_2} \exp\int_{0}^{t}\ph_2(f_sx_2)\,ds\right).
	\end{align*}
	Now, if $E_t \subset \CCC_t$ and $F_t\subset \DDD_t$ are minimal $(t,\eps)$ spanning sets, then $E_t\x F_t$ is a $(t,\eps)$-spanning set for $(\CCC\x\DDD)_t$. It follows that 
	\begin{align*}
	P(\CCC\x\DDD,\Phi;\FFF\x \FFF) &= \lim\limits_{\eps\to 0}\limsup_{t\to\infty}\frac{1}{t}\log\Lambda_t^{\text{span}}(\CCC\x\DDD,\Phi,\eps;\FFF\x \FFF)
	\\
	&\leq \lim\limits_{\eps\to 0}\limsup_{t\to\infty}\frac{1}{t}\log\Lambda_t^{\text{span}}(\CCC,\ph_1,\eps;\FFF)\Lambda_t^{\text{span}}(\DDD,\ph_2,\eps;\FFF)
	\\
	&\leq P(\CCC,\ph_1;\FFF) + P(\DDD,\ph_2;\FFF).
	\end{align*}
	The reverse inequality does not hold in general, as $\limsup$ is not superadditive. However, if $\CCC = \DDD$ and $\ph_1 = \ph_2 = \ph$, the inequality does hold. Given a maximal $(t,\eps)$-separating set $E$ for $\CCC_t$, then $E\x E$ is a $(t,\eps)$-separating set for $\CCC_t\x\CCC_t$. Hence, 
	$$\Lambda_t(\CCC\x\CCC,\Phi,\eps;\FFF\x \FFF) \geq \Lambda_t(\CCC,\ph,\eps;\FFF)^2.$$
	Then we have that for all $\eps > 0$,
	\begin{align*}
	P(\CCC\x\CCC,\Phi,\eps;\FFF\x \FFF) &= \limsup_{t\to\infty}\frac{1}{t}\log \Lambda_t(\CCC\x\CCC,\Phi,\eps;\FFF\x \FFF)
	\\
	&\geq 2\limsup_{t\to\infty}\frac{1}{t}\log \Lambda_t(\CCC,\ph,\eps;\FFF) = 2P(\CCC,\ph,\eps;\FFF). \qedhere
	\end{align*}
\end{proof}

\subsection{$\lambda$-decompositions} \label{lambdaestimate}
Recall Definition \ref{def.decomp} of a \emph{decomposition} of the space of orbit segments. We are interested in decompositions where $\GGG$ has specification and the Bowen property, and the pressure of $\PPP, \SSS$ is less than the whole space. Given a decomposition $(\PPP, \GGG, \SSS)$ for $(X, \FFF)$, we need to find a decomposition for the product system $(X \times X, \FFF \times \FFF)$ with nice properties. We make the following definition.
\begin{defn}
	Let $X$ be a compact metric space, $\FFF : X\to X$ a continuous flow, and $\ph : X\to\R$ a continuous potential. Let $\lambda : X\to[0,\infty)$ be a bounded lower semicontinuous function and $\eta>0$.  Let $B(\eta) = \{(x,t)\mid \frac{1}{t}\int_{0}^{t}\lambda(f_s(x))\,ds < \eta\}$ and $$\GGG(\eta) = \{(x,t)\mid \frac{1}{\rho}\int_{0}^{\rho}\lambda(f_s(x))\,ds\geq \eta \text{ and }\frac{1}{\rho}\int_{0}^{\rho}\lambda(f_{-s}f_t(x))\,ds\geq \eta \text{ for }\rho \in [0, t] \}.$$
Let $\PPP = \SSS = B(\eta)$, and let $\GGG = \GGG(\eta)$.	We define a decomposition $(\PPP, \GGG, \SSS)$ as follows. Given an orbit segment $(x,t)\in X\x[0, \infty)$, we decompose $(x,t)$ by taking the longest initial segment in $\PPP$ as the prefix, and the longest terminal segment which lies in $\SSS$ as the suffix. The good core is what is left over.  We say that a decomposition $(\PPP, \GGG, \SSS)$ defined in this way is a \emph{$\lambda$-decomposition} (with constant $\eta$).
\end{defn}
We ask that the function $\lambda$ is bounded and lower semi-continuous  since this allows both continuous functions as well as indicator functions of open sets. The decompositions used to study rank one geodesic flow in \cite{BCFT} are $\lambda$-decompositions, using the continuous function $\lambda$ defined in  \S \ref{geoflow}. The decompositions used in \cite{CFT2, CFT} to study equilibrium states for the Ma\~n\'e and Bonatti-Viana classes of DA systems can be taken to be $\lambda$-decompositions, where $\lambda$ is the indicator function of the complement of the small closed ball(s) where the original Anosov dynamics where perturbed. We note that the decompositions that were used to study $\beta$-shifts and $S$-gap shifts in \cite{CT} are defined combinatorially and are not $\lambda$-decompositions.

We describe pressure estimates for a $\lambda$-decomposition. Recall that given $\CCC\subset X\times [0, \infty)$, the collection $[\CCC]$ is  given by:
\[
[\CCC] := \{(x,n) \in X\times \mathbb N \mid (f_{-s}x, n+s+t) \in \CCC \text{ for some }s,t\in [0,1]\}.
\]
For a general decomposition, it is formally necessary to consider collections $[\PPP], [\SSS]$ in place of $\PPP, \SSS$ at a technical stage of the proof in \cite{CT} where a summation argument on growth of partition sums is required. However, for a $\lambda$-decomposition,  this distinction does not matter due to the following lemma.

\begin{lem}\label{BracketRem}
For all $\eps > 0$, $P([B(\eta)],\ph)\leq P(B(\eta + \eps),\ph)$.
\end{lem}

\begin{proof}
Let $\eta > 0$. For all $n\in\N$, we show that if $(x,n)\in [B(\eta)]$, then $(x,n)\in B((\frac{n+2}{n})\eta)$.
Observe that if $(x,n)\in [B(\eta)]$, then there exist $t,s\in [0,1]$ such that 
$$(n+t+s)\eta > \int_{-s}^{n+t} \lambda(f_rx)\,dr = \int_{-s}^0\lambda(f_rx)\,dr + \int_{0}^{n}\lambda(f_rx)\,dr + \int_n^{n+t}\lambda(f_rx)\,dr.$$
Since $\lambda \geq 0$, we see that
$\frac{1}{n}\int_{0}^{n}\lambda(f_rx)\,dr < \frac{(n+t+s)\eta}{n} \leq \left(\frac{n+2}{n}\right)\eta.$
Thus, given $\eps>0$ and any large $n$, we have $[B(\eta)]_n \subset B(\eta + \eps)_n$. The pressure estimate follows.
\end{proof}

We have the following pressure estimates for $\lambda$-decompositions.
\begin{thm}\label{Pressure Decrease}
	If the entropy map is upper semicontinuous, then	
	\[
	\lim\limits_{\eta\to 0}P(B(\eta),\ph) \leq \sup_{\mu\in\MMM(X,\FFF)} \{ P_\mu(\ph) : \int \lambda\,d\mu=0\}. 
	\]
\end{thm}
The interesting case in the above theorem is when $\{ \mu : \int\lambda\,d\mu = 0\} \neq \emptyset$. This can only fail if $B(\eta)$ has orbit segments of bounded length for small enough $\eta$. In that case, the inequality still holds, interpreting both sides as $-\infty$.
\begin{proof}
	We assume $\{ \mu : \int\lambda\,d\mu = 0\} \neq \emptyset$. For all $\eta \geq 0$, define $\MMM_\lambda(\eta) = \{\mu\in \MMM_\FFF(X)\mid \int \lambda\,d\mu \leq \eta\}$. Recalling Definition \ref{empirical}, we claim that for all $\mu\in \MMM(B(\eta))$, we have $\int \lambda\,d\mu \leq \eta$. First, consider an arbitrary empirical measure $\s{E}_{(x,t)}$, where $(x,t)\in B(\eta)$. We have
	$\int\lambda\,d\s{E}_{(x,t)} = \frac{1}{t}\int_{0}^{t}\lambda(f_sv)\,ds \leq \eta.$ For any convex combination of such measures, $\mu_t$, it follows that $\int\lambda\,d\mu_t \leq \eta$. 
	
	Therefore, for any sequence of measures $(\mu_{t_k})$ that converges to $\mu\in \MMM(B(\eta))$, by lower semicontinuity of $\lambda$, we have that 
	$\int\lambda\,d\mu \leq  \liminf \int\lambda\,d\mu_{t_k} \leq \eta.$
	Hence, we have shown that $\MMM(B(\eta))\subset \MMM_\lambda(\eta)$. Therefore, by Proposition \ref{bcftest}, we have shown that 
	$$P(B(\eta),\ph)\leq \sup_{\mu\in \MMM(B(\eta))}P_\mu(\ph)\leq \sup_{\mu\in \MMM_\lambda(\eta)}P_\mu(\ph).$$
	Additionally, this proof shows that for all $\eta$, we have $\MMM_\lambda(\eta)$ is compact.

	Now, observe that $\MMM_\lambda(0) = \bigcap_{\eta > 0}\MMM_\lambda(\eta)$ and let $\eps > 0$. By compactness and upper semicontinuity of the entropy map, for sufficiently small $\eta$, we have that $P_\mu(\ph)\leq P_\nu(\ph) + \eps$ for all $\mu\in \MMM_\lambda(\eta)$ and $\nu\in \MMM_\lambda(0)$. Thus for sufficiently small $\eta$, we have
$P(B(\eta),\ph) \leq \sup_{\mu\in \MMM_\lambda(\eta)}P_\mu(\ph) \leq \sup_{\mu\in \MMM_\lambda(0)}P_\mu(\ph) + \eps$. \qedhere
\end{proof}
By Lemma \ref{BracketRem}, it thus follows that $\lim\limits_{\eta\to 0}P([B(\eta)],\ph) \leq \sup \{ P_\mu(\ph) : \int \lambda =0\}.$

\subsection{Products of $\lambda$-decompositions} We want to find a decomposition for a product system $(X\x X,\FFF\x \FFF)$. When $(\PPP, \GGG, \SSS)$ is a $\lambda$-decomposition for $(X, \FFF)$, we are able to find a related decomposition on the product system as follows. We define $\tilde \lambda : X \x X \to [0, \infty)$ by
\[
\tilde{\lambda}(x,y) = \lambda(x)\lambda(y).
\]
This function inherits boundedness and lower semicontinuity from $\lambda$, and we consider a $\tilde \lambda$-decomposition for $(X\x X,\FFF\x \FFF)$. That is, for  $\eta>0$, we let 
\[
\tilde B(\eta) = \{((x,y),t)\mid \frac{1}{t}\int_{0}^{t}\tilde \lambda(f_sx, f_sy)\,ds < \eta\},
\] 
and we let $\tilde \GGG(\eta)$ be the set of orbit segments $((x,y),t)$ such that
\[
\frac{1}{\rho} \int_{0}^{\rho}\tilde \lambda(f_sx, f_sy)\,ds\geq \eta, \frac{1}{\rho}\int_{0}^{\rho}\tilde \lambda(f_{t-s}x, f_{t-s}y)\,ds\geq  \eta \text{ for all }\rho \in [0, t].
\]
The collections $\tilde \PPP = \tilde \SSS = \tilde B(\eta)$, and $\tilde \GGG = \tilde \GGG(\eta)$ define a $\tilde \lambda$-decomposition $(\tilde \PPP, \tilde \GGG, \tilde \SSS)$ for $(X\x X,\FFF\x \FFF)$.
		\begin{lem} \label{subsetprod}
			Let $M\in[0, \infty)$. For $0\leq \eta \leq 1$, let $(\PPP, \GGG, \SSS)$ and $(\tilde \PPP, \tilde \GGG, \tilde \SSS)$ be the $\lambda$-decomposition with constant $\eta\norm{\lambda}^{-1}$ and $\tilde \lambda$-decomposition with constant $\eta$  respectively. Then 
			 $\tilde \GGG^M \subset \GGG^M \x\GGG^M$. 
		\end{lem}
		\begin{proof}
			Let $((x,y),t)\in \tilde \GGG^M= \tilde \GGG^M(\eta)$, with a prefix of length $m_1$ and a suffix of length $m_2$. We show that $(x,t)\in \GGG^M = \GGG^M(\eta \norm{\lambda}^{-1})$. To do this, we need to show that the prefix in the $\lambda$-decomposition of $(x,t)$ has length at most $m_1$, and the suffix has length at most $m_2$.  Observe that for all $0\leq r_1 < r_2 \leq t$ we have
			$$\int_{r_1}^{r_2}\lambda(f_sx)\,ds \geq \int_{r_1}^{r_2}\lambda(f_sx)\frac{\lambda(f_sy)}{\norm{\lambda}}\,ds = \frac{1}{\norm{\lambda}}\int_{r_1}^{r_2}\tilde{\lambda}(f_sx, f_s y)\,ds.$$
		
			Therefore, we see that for all $r > m_1$, we have
			$$\frac{1}{r}\int_{0}^{r}\lambda(f_sx)\,ds \geq \frac{1}{\norm{\lambda}}\frac{1}{r}\int_{0}^{r}\tilde{\lambda}(f_s x,f_s y)\,ds \geq \frac{\eta}{\norm{\lambda}}.$$
			Hence, the prefix of $(x,t)$ is of length at most $m_1$. A similar proof shows that the suffix is of length at most $m_2$. The same argument applies to $(y,t)$ and so we conclude that 
			$((x,y),t)\in \GGG^M \times \GGG^M$.
		\end{proof}
		The following corollary is immediate from applying Lemmas \ref{specprod} and \ref{bowenprod}.
		\begin{cor} \label{specbowlift}
		Let $(\PPP, \GGG, \SSS)$ and $(\tilde \PPP, \tilde \GGG, \tilde \SSS)$ be as above. If $\GGG^M$ has specification, then so does $\tilde \GGG^M$. If $\GGG$ has the Bowen property for a function $\varphi$, then $\tilde \GGG$ has the Bowen property for $\Phi(x,y) = \ph(x)+ \ph(y)$.
		\end{cor}

		\section{almost expansive and product expansive measures} \label{prodalmost} 
	In this section, we introduce a new notion called product expansivity  that plays a crucial role in our proofs. We collect approximation and counting properties for almost expansive and product expansive measures.

\subsection{Approximation lemma for almost expansive measures}
Recall that for a flow $(X, \FFF)$, the non-expansive set $\NE(\eps)$ was defined in \S \ref{pressure}, and a measure $\nu \in \MMM(X , \FFF )$ is \emph{almost expansive at scale $\eps$} if $\nu(\NE(\eps)) = 0$. Given  $t > 0$, we write $f_{[-t,t]}A = \{f_sA\mid s\in [-t,t]\}$  for a  flow-out of a set. Given a $(t,\eps)$-separated set $E$ of maximal cardinality, a partition $\s{A}$ is \emph{adapted} to $E$ if for all $A\in\s{A}$, there exists $x\in E$ such that $B_t(x,\frac{\eps}{2}) \subset A \subset \overline{B_t}(x,\eps)$. The following proposition 
has a similar spirit to \cite[Lemma 2]{Bow74}.
	\begin{prop}\label{InvariantApprox}
		Let $\FFF$ be a continuous flow on a compact metric space $X$, and suppose $\nu\in M_\FFF(X)$ is almost expansive at scale $\eps$. Let $s > 0$. Let $\alpha > 0$. Then for sufficiently small $\rho \in (0, \eps/2)$,  the following holds true. Let $\s{A}_t$ be an adapted partition for a $(t,\rho)$-separated set of maximal cardinality. Let $A\subset X$ be a positive measure set. Then for each $\kappa > 0$, there exists $t_0$ such that if $t\geq t_0$, then we can find $U\subset\s{A}_t$ such that $\nu(f_{t/2}U\setminus f_{[-3s,3s]}A) < \kappa$ and $\nu(A\setminus f_{t/2}U) < \alpha$.
	\end{prop}
	\begin{proof}
		Let $X_{s,\gamma} = \{x\mid \Gamma_\gamma(x)\subset f_{[-s,s]}(x)\}$.  First, we show that $\bigcup_{\gamma < \eps}X_{s,\gamma} \supset X\setminus \NE(\eps)$. Let $x\in X\setminus \NE(\eps)$. Then, there exists a minimal $r > 0$ such that $\Gamma_\eps(x)\subset f_{[-r,r]}(x)$. If $r\leq s$, then $x\in X_{s,\eps}$, so we assume $r > s$. Now note that the minimality of $r$ implies that $f_tx\neq x$ for $t\in [-r,r]$. Consequently, $f_{[-r,-s]}(x)\cup f_{[-s,-r]}(x)$ is compact and disjoint from $\{x\}$, so there exists $\gamma$ less than the distance between these two sets. Then $\Gamma_\gamma(x)\subset f_{(-s,s)}(x)$, and we have shown the desired inclusion.

		By almost expansivity, there exists $\gamma$ such that $\nu(X_{s,\gamma}) > 1 - \frac{\alpha}{3}$. Now let $\rho < \min \{\gamma,\frac{\eps}{2}\}$ be arbitrary, and write $X_s := X_{s,\rho}$. For $A\subset X$ and $t > 0$, define
		$$\operatorname{diam}_{[-s,s]}A = \sup_{x_1,x_2\in A}\inf_{t_1,t_2\in [-s,s]}d(f_{t_1}x_1,f_{t_2}x_2).$$		
	As $\bigcap_t B_{[-t,t]}(x,\rho)\subset f_{[-s,s]}(x)$, for each $x\in X_s$, we have $\operatorname{diam}_{[-s,s]}B_{[-t,t]}(x,\rho)\to 0$ as $t\to\infty$. Now let $\s{A}_t' = f_{t/2}\s{A}_t$ and set $w_t(x)$ to be the element of $\s{A}_t'$ containing $x$. By construction $w_t(x)\subset \overline{B}_{[-t/2,t/2]}(x,2\rho)$, and so we have that $\operatorname{diam}_{[-s,s]}w_t(x)\to 0$ as $t\to\infty$ for $\nu$-a.e. $x\in X_s$. By Egorov's theorem, there exists $X_s'\subset X_s$ with $\nu(X_s\setminus X_s') < \frac{\alpha}{3}$ such that this convergence is uniform on $X_s'$.

		Now let $A' = A\cap X_s'$. Then define $K_1\subset A'$ and $K_2\subset X\setminus f_{[-3s,3s]}A$ to be compact such that $\nu(A'\setminus K_1) < \frac{\alpha}{3}$ and $\nu(X\setminus (f_{[-3s,3s]}A\cup K_2)) < \kappa$. Now consider $f_{[-s,s]}K_1$ and $f_{[-s,s]}K_2$. These are compact and disjoint, because $f_{[-s,s]}K_1\subset f_{[-s,s]}A$ and $f_{[-s,s]}K_2\subset X\setminus f_{(-2s,2s)}A$. Therefore, they are uniformly separated by some distance $\theta > 0$.
		Consequently,
		$$\inf_{t_1,t_2\in [-s,s]}\{d(f_{t_1}x_1,f_{t_2}x_2)\mid x_1\in K_1, x_2\in K_2\} \geq \theta.$$
		By uniform convergence on $X_s'$, there exists $t_0\in[0,\infty)$ such that $\operatorname{diam}_{[-s,s]}w_t(x) < \theta$ for every $t\geq t_0$ and $x\in X_s'$. Therefore, for all $t\geq t_0$, if $w\in \s{A}_t'$ satisfies $w\cap K_1\neq \es$, then $w\cap K_2 = \es$. Thus defining $U' = \bigcup\{w\in\s{A}_t'\mid w\cap K_1 \neq \es\}$, observe that $K_1\subset U'$ and $K_2\cap U' = \es$. Hence, we see that 
		$$\nu(A\setminus U') \leq \nu(A\setminus K_1) \leq \nu(A\setminus A') + \nu(A'\setminus K_1) < \alpha$$
		and
		$$\nu(U'\setminus f_{[-3s,3s]}A) \leq \nu(X\setminus (f_{[-3s,3s]}A\cup K_2)) < \kappa.$$
		If we set $U\subset \s{A}_t$ to be $U = f_{-t/2}U'$, then we are done.
	\end{proof}

\subsection{Product expansive measures}
Consider a product of flows $(X \x X, \FFF \x \FFF)$, and define the bi-infinite Bowen ball to be
\begin{equation} \label{Gammaeps}
\Gamma_\eps(x,y) = \{(x',y')\in X\x X\mid \tilde d((f_tx, f_ty), (f_tx', f_ty')) < \eps \text{ for all } t \in \RR\}.
\end{equation}
		
We can also write $\Gamma_\eps(x,y)$ as $\Gamma_\eps((x,y); \FFF \x \FFF, \tilde d)$ when we want to emphasize the metric and the dynamics.
\begin{defn}
	The set of \emph{product non-expansive points at scale $\eps$} is
	$$\operatorname{NE^\x}(\eps) := \{(x,y)\in X\x X\mid \Gamma_\eps(x,y)\not\subset f_{[-s,s]}(x)\x f_{[-s,s]}(y)\text{ for any }s > 0\}.$$
\end{defn}
\begin{defn}
 We say a measure $\nu \in \MMM(X \x X, \FFF \x \FFF)$ is \emph{product expansive at scale $\eps$} if $\nu(\NE^\x(\eps)) = 0$.
\end{defn}
We have the following  basic lemma. 
\begin{lem} \label{NEx}
We have $\NE^\x(\eps) = (X \x \NE(\eps)) \cup (\NE(\eps) \x X)$.
\end{lem}
\begin{proof}
The claim follows from showing that the complements are equal, using that $\Gamma_\eps(x,y) = \Gamma_\eps(x)\x \Gamma_\eps(y)$, and so $\Gamma_\eps(x,y)\subset f_{[-s,s]}(x)\x f_{[-s,s]}(y)$ for some $s > 0$ if and only if $\Gamma_\eps(x)\subset f_{[-s,s]}(x)$ and $\Gamma_\eps(y)\subset f_{[-s,s]}(y)$.
\end{proof}
It can be checked easily using Lemma \ref{NEx} that if $\nu\in \MMM(X, \FFF)$ is almost expansive at scale $\eps$, then $\nu\x\nu$ is product expansive at scale $\eps$. Recall that for an invertible discrete-time dynamical system $(X, f)$,  we say that a measure $\nu \in \MMM(X, f)$ is \emph{almost entropy expansive} at scale $\eps$ in a metric $d$ if $h(\Gamma_\eps(x;f, d))=0$ for $\nu$-a.e. $x \in X$, where 
\[
\Gamma_\eps(x;f, d) = \{x'\in X \mid d(f^nx, f^nx')\leq \eps \text{ for all } n \in \ZZ\},
\]
and $h(\cdot)= P(\cdot,0)$ is the topological entropy. 
\begin{prop}
	If $\nu\in \MMM(X \x X, \FFF \x \FFF)$ is product expansive at scale $\eps$, then $\nu$ is almost entropy expansive at scale $\eps$ with respect to the time-t map $f_t\x f_t$ and metric $\tilde d_t$.
\end{prop}
\begin{proof} 
Observe that $\Gamma_\eps((x,y);f_t\x f_t, \tilde d_t) = \Gamma_\eps ((x,y); \FFF \x \FFF, \tilde d)$. Since $\nu$ is product expansive, it follows that for $\nu$-a.e. $(x,y)$, 
\[
\Gamma_\eps((x,y);,f_t\x f_t, \tilde d_t) \subset f_{[-s,s]}(x)\x f_{[-s,s]}(y)
\] 
for some $s=s(x,y)\in [0,\infty)$. By Proposition \ref{Addition}, $h(f_{[-s,s]}(x)\x f_{[-s,s]}(y)) \leq h(f_{[-s,s]}(x)) + h(f_{[-s,s]}(y))$. Any finite orbit segment has zero entropy, see for example the proof of \cite[Proposition 3.3]{CT}. It follows that for $\nu$-a.e. $(x,y)\in X\x X$, 
\[
h(\Gamma_\eps((x,y);f_t\x f_t, \tilde d_t)) \leq h(f_{[-s,s]}(x)\x f_{[-s,s]}(y))  \leq h(f_{[-s,s]}(x)) + h(f_{[-s,s]}(y))= 0,
\] 
and thus $\nu$ is almost entropy expansive at scale $\eps$ in the metric $\tilde d_t$ with respect to the map $f_t \x f_t$.
\end{proof}
It is shown in \cite[Theorem 3.2]{CT} that if $\nu$ is almost entropy expansive at scale $\eps$, then every partition $\AAA$ with diameter at most $\eps$ has $h_{\nu}(f) = h_{\nu}(f, \AAA)$. Thus, we have the following corollary.

\begin{cor}\label{ScaleEntropy}
If $\nu\in \MMM(X \x X, \FFF \x \FFF)$ is  product expansive at scale $\eps$ and $\s{A}_t$ is a partition adapted to a maximal cardinality $(t,\eps/2)$-separated set $E_t$, then $h_{\nu}(f_t\x f_t,\s{A}_t) = h_\nu(f_t\x f_t)$.
\end{cor}

We have the following approximation result for product expansive measures, which generalizes \cite[Proposition 3.10]{CT}, building again on \cite[Lemma 2]{Bow74}. 

\begin{prop}\label{Approximation}
	Let $\FFF$ be a continuous flow on a compact metric space $X$, and suppose $\nu\in \MMM(X\x X, \FFF \x \FFF)$ is product expansive at scale $\eps$. Let $\gamma\in (0,\eps/2)$, and for each $t > 0$, let $\s{A}_t$ be an adapted partition for a $(t,\gamma)$-separated set of maximal cardinality. Let $Q\subset X\x X$ be a measurable set invariant under the $\R^2$-action, meaning for all $t,s\in\R$, $(f_t\x f_s)Q = Q$. Then for every $\alpha > 0$, there exists $t_0$ so that if $t\geq t_0$, we can find $U\subset \s{A}_t$ such that $\nu(U\vartriangle Q) < \alpha$.
\end{prop}
Note that  Proposition \ref{Approximation} does not apply for sets $Q$ that are $\FFF \times \FFF$ invariant i.e. invariant for each map $f_t \x f_t$. We need $Q$ to be invariant for EVERY map $f_s \x f_t$ where $s,t \in \RR$.
\begin{proof}
	We will assume $\nu(Q) > 0$. For $w\subset X\x X$ and $s\in[0,\infty)$, define
	$$\operatorname{diam}_{[-s,s]}w = \sup_{(x_1,y_1),(x_2,y_2)\in w}\inf_{t_1,t_2\in [-s,s]}\max\{d(f_{t_1}x_1,f_{t_2}x_2),d(f_{t_1}y_1,f_{t_2}y_2)\}.$$
	Now, for $s\in[0,\infty)$, define $X_s = \{x\mid\Gamma_\eps(x)\subset f_{[-s,s]}(x)\}$, and set
	$$\tilde{X_s} = \{(x,y)\mid\Gamma_\eps(x,y)\subset f_{[-s,s]}(x)\x f_{[-s,s]}(y)\} = X_s\x X_s.$$
	Now fix $\beta > 0$. Observe $\bigcup_s X_s = X\setminus \operatorname{NE}(\eps)$. Consequently, $\bigcup_s \tilde{X_s} = (X\x X)\setminus \operatorname{NE}^\x(\eps)$, and so there exists $s$ such that $\nu(\tilde{X_s}) > 1 - \beta$. Furthermore, for every $(x,y)\in \tilde{X_s}$, we have that
	\begin{align*}
	\operatorname{diam}_{[-s,s]}B_{[-t,t]}((x,y),\eps) &= \operatorname{diam}_{[-s,s]}(B_{[-t,t]}(x,\eps)\x B_{[-t,t]}(y,\eps))
	\\
	&= \max\{\operatorname{diam}_{[-s,s]}B_{[-t,t]}(x,\eps),\operatorname{diam}_{[-s,s]}B_{[-t,t]}(y,\eps)\},
	\end{align*}
	which tends to $0$. Now let $\s{A}_t' = (f_{t/2}\x f_{t/2})\s{A}_t$, and write $w_t(x,y)$ for the element of the partition $\s{A}_t'$ which contains $(x,y)$. Observe that for each $(x,y)\in X\x X$, there exists a point $(x',y')$ such that $w_t(x,y)\subset \overline{B}_{[-t/2,t/2]}((x',y'),\gamma)$. Therefore, $w_t(x,y)\subset \overline{B}_{[-t/2,t/2]}((x,y),2\gamma)$. Thus, $\operatorname{diam}_{[-s,s]}w_t(x,y)\to 0$ for almost every $(x,y)\in\tilde{X_s}$. By Egorov's theorem, there exists $\tilde{X'_s}\subset \tilde{X_s}$ with $\nu(\tilde{X_s}\setminus\tilde{X'_s}) < \beta$ such that convergence is uniform on $\tilde{X'_s}$. Now set $Q' = \tilde{X'_s}\cap Q$, and let $K_1\subset Q'$ and $K_2\subset (X\x X)\setminus Q$ be compact with $\nu(Q'\setminus K_1) < \beta$ and $\nu((X\x X)\setminus (Q\cup K_2)) < \beta$. For $i = 1,2,$ define
	$$K_i^s = \{(f_{t_1}(x),f_{t_2}(y))\mid (x,y)\in K_i, t_1,t_2\in [-s,s]\}.$$
	Then $K_i^s$ is compact, and $K_1^s\subset Q$ and $K_2^s\subset (X\x X)\setminus Q$. Thus, there exists $\delta > 0$ such that $d(K_1^s, K_2^s) \geq \delta$ by compactness. So, for all $(x_i,y_i)\in K_i$,
	$$\inf_{t_i,r_i\in [-s,s]}\max\{d(f_{t_1}(x_1),f_{t_2}(x_2)),d(f_{r_1}(y_1),f_{r_2}(y_2))\} \geq \delta.$$
	Now uniform convergence on $Q'$ implies that there exists $t_0$ such that for all $t\geq t_0$, $\operatorname{diam}_{[-s,s]}w_t(x,y) < \delta$ for all $(x,y)\in Q'$. Hence, for all $t\geq t_0$, if $w\in \s{A}_t'$ and $w\cap K_1\neq\es$, then $w\cap K_2 = \es$. Therefore, setting $U' = \bigcup\{w\in\s{A}_t'\mid w\cap K_1 \neq \es\}$, we have that $K_1\subset U'$ and $K_2\cap U' = \es$, and so,
	\begin{align*}
	\nu(U'\vartriangle Q) &= \nu(U'\setminus Q) + \nu(Q\setminus U')
	\\
	&\leq \nu((X\x X)\setminus (Q\cup K_2)) + \nu(Q\setminus K_1)
	\\
	&\leq \beta + \nu(Q\setminus Q') + \nu(Q'\setminus K_1) \leq \beta + 2\beta + \beta.
	\end{align*}
	As we can choose $\beta$ to be arbitrarily small, we have that $\nu(U'\vartriangle Q) < \alpha$. Therefore, defining $U\subset \s{A}_t$ by $U = (f_{-t/2}\x f_{-t/2})U'$, we see that
	$$\nu(U\vartriangle Q) = \nu((f_{-t/2}\x f_{-t/2})(U\vartriangle Q)) = \nu(U'\vartriangle Q) < \alpha. \qedhere $$
\end{proof}
		
\subsection{Counting estimates}
We will require a technical counting lemma from \cite{CT}. In our setting, the statements of Lemma 4.8 and Lemma 4.18 of \cite{CT} easily combine to give the following statement. We refer to \cite{CT} for the proofs.

\begin{lem}\label{CombinedLemma}
Let $(\PPP,\GGG,\SSS)$ be a decomposition for $X \times [0, \infty)$ such that
\begin{enumerate}
	\item $\GGG$ has specification at all scales;
	\item $\ph$ has the Bowen property on $\GGG$ and
	\item  $P( [\PPP]  \cup  [\SSS], \ph) < P(\ph)$,
\end{enumerate}
Fix sufficiently small $\gamma \in (0,\eps/4)$ where $\eps$ satisfies $P(\ph,\eps) = P(\ph)$. Then for every $\alpha\in(0,1)$, there exists a constant $C_\alpha > 0$ and $M\in [0, \infty)$ such that that for sufficiently large $t$, the following is true. Consider an equilibrium state $\nu$ for $\ph$ and a family $\{E_t\}_{t > 0}$ of maximizing $(t,\gamma)$-separated sets for $\Lambda_t(X, \ph, \gamma)$. Let $\{\mathscr{A}_t\}$ be adapted partitions for the family $\{E_t\}$, and given $x\in E_t$ let $w_x$ denote the corresponding partition element in $\mathscr{A}_t$.  If $h_\nu(f_t,\s{A}_t) = h_\nu(f_t)$ and $E'_t\subset E_t$ satisfies $\nu(\bigcup_{x\in E_t'}w_x)\geq \alpha$, then if we write $\CCC = \{(x,t)\mid x\in E_t'\}$, we have
$$\Lambda_t(\CCC\cap \GGG^M, \ph, \gamma) \geq C_\alpha e^{tP(\ph)}.$$
\end{lem}

We will apply the above lemma both when $X= T^1M$ and $X = T^1M \x T^1M$. In the latter case, the hypotheses on the scale $\eps$ and entropy of a partition are satisfied for sufficiently small scales by entropy expansivity and Corollary \ref{ScaleEntropy}.

	\section{Pressure estimate for $(T^1M \times T^1M, (g_t \times g_t))$} \label{pressureest}

In this section, we assume the hypotheses of Theorem \ref{thmA}. In particular, we assume that $\ph: T^1M \to \RR$ is H\"older or $q\ph^u$ for some $q \in \RR$, $P(\Sing,\ph) < P(\ph)$ and that $\eps$ is chosen so that any equilibrium state is almost expansive at scale $\eps$. By \cite[Lemma 5.3]{BCFT}, any $\eps$ less than a third of the injectivity radius is small enough. We also assume that $\Sing\neq \es$. In the case where this does not occur, the results hold by simpler arguments. The potential $\Phi: T^1M \x T^1M \to \RR$ is given by $\Phi(x,y)= \ph(x)+\ph(y)$.	
	
	Recall that $\lambda: T^1M \to [0, \infty)$ is the function that measures the smallest curvature of the horospheres through a point, defined in \S \ref{geoflow}, and that $\tilde \lambda: T^1M \x T^1M \to [0, \infty)$ is defined by $\tilde \lambda(x, y) = \lambda(x)\lambda(y)$. We have specification and the Bowen property for $\ph$ on $\GGG$ for the $\lambda$-decomposition of $(T^1M,(g_t))$ due to \cite[Theorem 4.1]{BCFT} and \cite[Corollaries 7.5,7.8]{BCFT} respectively. Applying Corollary \ref{specbowlift}, we have  specification and the Bowen property for $\Phi$ on $\tilde \GGG$ for the $\tilde \lambda$-decomposition of $(T^1M \x T^1M,(g_t \x g_t))$.
	
In light of Theorem \ref{Pressure Decrease}, the $\tilde \lambda$-decomposition will be useful if we can control $\sup \{ P_\nu(\Phi) : \int \tilde \lambda\,d\nu =0\}$. Note that $\{ \nu: \int \tilde \lambda d \nu=0\} \neq \emptyset $ because for any $m$ supported on $\Sing$, we have $\int \tilde \lambda\, d(m \x m) =0$. We prove the following proposition.
\begin{prop} \label{upperboundfreeenergy}
Let $\varphi : T^1M\to \R$ be continuous, then $\Phi: T^1M \x T^1M \to \R$ satisfies
\[
\sup \{ P_\nu(\Phi) : \int \tilde \lambda d \nu=0\} \leq P(\Sing, \ph)+ P(\ph).
\]
\end{prop}
\begin{proof} 
Suppose $\nu$ satisfies $\int \tilde \lambda \,d \nu =0$. If $\nu(\Reg \x \Reg)>0$, then we would have a recurrence set $A$ of positive measure with $A\subset \Reg \x \Reg$. Since $A$ is forwards and backwards recurrent, then for any $(x,y)\in A$, $d(g_tx,\Sing)$ does not converge to $0$ for $t\to\infty$, nor for $t\to -\infty$. It follows that $\lambda(x) > 0$ by Corollary 3.5 of \cite{BCFT}. The same argument shows that $\lambda(y) > 0$. Thus, $\tilde{\lambda}(x,y) > 0$ for all $(x,y) \in A$. This would imply that $\int \tilde \lambda \,d \nu >0$, which is a contradiction. Consequently, the complement of $\Reg \x \Reg$ in $T^1M \x T^1M$ carries full $\nu$-measure. In other words,
\[
\nu((\Sing \x T^1M) \cup (T^1M \x \Sing))=1.
\]
The set  $(\Sing \x T^1M) \cup (T^1M \x \Sing)$ is compact and invariant, so we can apply the variational principle. It follows, together with an application of Lemma \ref{Maximum} and Proposition \ref{Addition}, that
\[
P_\nu(\Phi) \leq P((\Sing \x T^1M) \cup (T^1M \x \Sing), \Phi) \leq P(\Sing,\ph) + P(\ph). \qedhere
\]
\end{proof}
Since $g_t$ is entropy expansive, $g_t\x g_t$ is entropy expansive, and thus the entropy map on $(T^1M\x T^1M, (g_t\x g_t))$ is upper semicontinuous.	
	Thus, combining Theorem \ref{Pressure Decrease} with Proposition \ref{upperboundfreeenergy}, and observing that $P(\Sing, \ph)+ P(\ph)< 2P(\ph)=P(\Phi)$, we have the following.
\begin{cor}
The collections $\tilde B(\eta)$ satisfy  $\lim\limits_{\eta\to 0}P(\tilde B(\eta),\Phi) <  P(\Phi)$. 
\end{cor}

The pressure estimate we need on the $\tilde \lambda$-decomposition is immediate from this and Lemma \ref{BracketRem}. 
\begin{cor}
If $\eta>0$ is chosen sufficiently small, then  the $\tilde \lambda$-decomposition at scale $\eta$, which we denote by  $(\tilde \PPP, \tilde \GGG, \tilde \SSS)$, satisfies $P([\tilde \PPP] \cup [ \tilde \SSS], \Phi)< P(\Phi)$.
	\end{cor}

This shows that we have the decomposition structure we need. We also verify that we have the expansivity property we require for an equilibrium state for $\Phi$.	
	\begin{prop} \label{prodexp}
Suppose that $P(\Sing, \ph)< P(\ph)$. Then any equilibrium state $\nu$ for $\Phi$ is product expansive.
\end{prop}
\begin{proof} 
First, we assume that $\nu$ is an ergodic equilibrium state. 
	 By Lemma 5.3 of \cite{BCFT},  for sufficiently small $\eps$, $\NE(\eps)\subset \Sing$. It thus follows from Lemma \ref{NEx} that
	\[
	\NE^\x(\eps) \subset (\Sing \times T^1M) \cup (T^1M \times \Sing).
	\]
Since $K:= (\Sing \times T^1M) \cup (T^1M \times \Sing)$ is $(g_t \x g_t)$-invariant, it has measure $1$ or $0$.
By Lemma \ref{Maximum}, Proposition \ref{Addition}, and our hypothesis,
$$P(K, \Phi) \leq P(\Sing, \ph)+ P(\ph)< 2P(\ph)=P(\Phi).$$
If $\nu(K)=1$, then by the variational principle, $P_\nu(\Phi) \leq P(K, \Phi) < P(\Phi)$, which contradicts $\nu$ being an equilibrium state for $\Phi$. Therefore, we must have $\nu(K)=0$. It follows that $\nu(\NE^\x(\eps))=0$, and so $\nu$ is product expansive.

Now suppose that $\nu$ is not ergodic. Then every measure in its ergodic decomposition is also an equilibrium state. By the argument above, each of these measures give measure $0$ to $\NE^\x(\eps)$. Thus, $\nu (\NE^\x(\eps))=0$.
\end{proof}

\section{Uniqueness of the equilibrium state on the product system} \label{productproof}
In this section, we continue to assume the hypotheses of Theorem \ref{thmA}. We write $\mu$ for the unique equilibrium state for $(T^1M,(g_t),\ph)$ provided by Theorem \ref{CT}.  We showed in Lemma \ref{ProductES} that  $\mu\x\mu$ is an equilibrium measure for $\Phi$. We show that $\mu\x\mu$  is ergodic and rule out any mutually singular equilibrium states.

	\subsection{Weak mixing for $\mu$} We show that the equilibrium state $\mu$ for $(g_t)$ is weak mixing using spectral techniques, adapting an idea from Bowen in \cite{rB72}. This is equivalent to ergodicity of $\mu \x \mu$. To carry out this strategy, we need to obtain a partial mixing estimate for ``flowed out'' positive measure sets. An estimate of this type appears in \cite{eF77} in the case of uniform specification, but it was not established in the non-uniform setting considered by \cite{CT}. Our argument is a sharpened version of the ergodicity proof in \cite{CT}. We use the following lemma, which is essentially Lemma 4.17 from \cite{CT}.

\begin{lem}\label{JointGibbs}
Assuming that $\FFF$ has the specification property on $\GGG^M$ with specification constant $\tau_M$ for all $M\in\R$, then for large $M$, there exists $Q_M' > 0$ such that for each $(x_1,t_1), (x_2,t_2)\in\GGG^M$ with $t_1, t_2\geq T(M)$ and each $q\geq 2\tau_M$, we have
$$\mu(B_{t_1}(x_1,\rho)\cap f_{-(t_1 + q)}B_{t_2}(x_2,\rho))\geq Q'_Me^{-(t_1 + t_2)P(\ph) + \int_{0}^{t_1}\ph(f_sx_1)\,ds \int_{0}^{t_2}\ph(f_sx_2)\,ds}.$$
\end{lem}
In \cite{CT}, the statement only gives the existence of some $q'$ satisfying this inequality in each interval $[q-2\tau_M,q]$. We are able to omit this condition because we are working with specification (i.e. exact transition times between orbit segments) as opposed to weak specification (i.e. an upper bound on the transition time).

\begin{prop}
		Let $s > 0$ be arbitrary. For all sets $A,B$ of positive measure, for large enough $t$, $\mu(f_{[-3s,3s]}A\cap f_tf_{[-3s,3s]}B) > 0$.
	\end{prop}
	\begin{proof}
		Let $P = f_{[-3s,3s]}A$ and $Q = f_{[-3s,3s]}B$ and let $2\alpha_1 = \min\{\mu(A),\mu(B)\}$. Now take $2\rho$ small enough so that we can apply Proposition \ref{InvariantApprox} with $\alpha$ as $\alpha_1$ and Lemma \ref{CombinedLemma}. Let $\s{A}_t$ be adapted partitions for $(t,2\rho)$-separated sets $E_t$ which are maximizing for $\Lambda_t(T^1M, \ph, 2\rho)$. Then we can take $U_t\subset\s{A}_t$ and $V_t\subset \s{A}_t$ such that $\liminf \mu(A\setminus f_{t/2}U_t) \leq \alpha_1$ and $\liminf \mu(B\setminus f_{t/2}V_t) \leq \alpha_1$. Thus, for large $t$, we have $\mu(U_t)= \mu(f_{t/2}U_t) \geq \alpha_1$, where the first equality is because the measure is invariant. We get the same estimate for $V_t$.

		From here, the argument from \cite[Proposition 4.19]{CT} applies using Lemma \ref{JointGibbs}, and using Lemma \ref{CombinedLemma} in place of Lemmas 4.18 and 4.8 of \cite{CT},  to show that for all $q > 2\tau_M$, where $\tau_M$ is the specification constant, and for all large $t$,
				$\mu(U_t\cap f_{-(t + q)}V_t) \geq \beta$
		for some constant $\beta > 0$. Now, by Proposition \ref{InvariantApprox}, choose $t_0$ such that for all $t\geq t_0$, we have
		$$\mu(f_{t/2}U_t \setminus P) < \frac{\beta}{2} \text{ and }\mu(f_{t/2}V_t\setminus Q) < \frac{\beta}{2}.$$
		Then, observe that for all $t,r > 0$, we have
		$$(f_{t/2}U_t\cap f_{-r}f_{t/2}V_t)\setminus (P\cap f_{-r}Q) \subset (f_{t/2}U_t\setminus P) \cup f_{-r}(f_{t/2}V_t\setminus Q).$$
		It follows that, writing $r := t + 3\tau_M$,
		\begin{align*}
		\mu(P\cap f_{-r}Q) &\geq \mu(f_{t/2}U_t\cap f_{-r}f_{t/2}V_t) - \mu(f_{t/2}U_t\setminus P) - \mu(f_{-r}(f_{t/2}U_t\setminus Q))
		\\
		&> \mu(U_t\cap f_{-r}V_t) - \beta \geq 0. \qedhere
		\end{align*} 
	\end{proof}

	We are ready to show that $\mu$ is weak mixing, and thus $\mu \times \mu$ is ergodic.
	\begin{thm}\label{Weak Mixing}
		The measure $\mu$ is weak mixing.
	\end{thm}
	
	\begin{proof}
		Suppose not. Then $\FFF$ has a non-zero eigenvalue $\theta$, and so there exists a non-constant $\ph\in L^2(\mu)$ and $\theta\neq 0$ such that for all $t$, 
		$\ph(f_tx) = e^{i\theta t}\ph(x)$  almost everywhere.
 		Applying Fubini's theorem as in \cite[\S 5.4]{Rok}, we can choose $\ph$ such that $\ph(f_tx) = e^{i\theta t}\ph(x)$ for all $x$ and $t$. Now, because $\ph$ is not constant, there exists a closed disk $B$ and $t_0$ such that $0 < \mu(\ph^{-1}(B)) < 1$ and $B\cap e^{i\theta t_0}B = \es$. For small $\beta > 0$, we have that
		$e^{i\theta[-\beta,\beta]}B \cap e^{i\theta[t_0 - \beta, t_0 + \beta]}B = \es$
		which implies that  
		$$f_{[-\beta,\beta]}\ph^{-1}(B)\cap f_{\frac{2\pi n}{\theta} + t_0}f_{[-\beta,\beta]}\ph^{-1}(B) = \es$$
		for all $n\in\ZZ$. Taking $s = \frac{\beta}{3}$ in the previous proposition, this is a contradiction, and our proof is complete.
	\end{proof}

	\subsection{No mutually singular equilibrium state}
		Recall that we have a decomposition $(\tilde \PPP, \tilde \GGG, \tilde \SSS)$ such that $\tGGG$ has the specification property and the Bowen property for $\Phi$, and $P([\tPPP] \cup [\tSSS], \Phi)< P(\Phi)$. We also have $P(\Phi, \gamma) = P(\Phi)$ for small $\gamma$ by entropy expansivity. We have the following lemma. 
		
			\begin{lem}\label{Gibbs}
				The measure $\mu \times \mu$ has the Gibbs property on $\tilde \GGG^M$.
			\end{lem}
			\begin{proof}
				Using the Gibbs property for $\GGG^M$ of $\mu$ from \cite[Lemma 4.16]{CT}, Lemma \ref{subsetprod}, and the fact that in our choice of metric $B_t((x,y),\rho)= B_t(x,\rho) \x B_t(y,\rho)$, we have
				\begin{align*}
				(\mu\x\mu)(B_t((x,y),\rho)) &= \mu(B_t(x,\rho))\mu(B_t(y,\rho))
				\\
				&\geq Q_M^2e^{-t2P(\ph) + \int_{0}^{t}\ph(f_sx)\,ds + \int_0^t\ph(f_sy)\,ds}
				\\
				&= Q_M^2e^{-tP(\Phi) + \int_{0}^{t}\Phi(f_sx,f_sy)\,ds} \qedhere
				\end{align*}
			\end{proof}

 Let $\nu$ be an ergodic equilibrium state such that $\nu\perp(\mu\x\mu)$. For all $t$, let $E_t$ be a maximizing $(t,\gamma)$-separated set for $\Lambda_t(T^1M \x T^1M, \Phi, \gamma)$, and let $\s{A}_t$ be an adapted partition for $E_t$, where $\gamma < \frac{\eps}{4}$ is chosen small enough to apply Lemma \ref{CombinedLemma}. Let $P$ be a set which satisfies $(g_{r_1}\x g_{r_2})P = P$ for all $r_1, r_2\in\R$,  $\nu(P) = 1$ and $(\mu\x\mu)(P)=0$.  Such sets exist. We show that the set $P$ defined to be the complement of  the set of generic points $G$ for $\mu\x\mu$ is such a set. By ergodicity, $(\mu\x\mu)(G) = 1$. Since $(\mu\x\mu)\perp\nu$, it follows that $\nu(G) = 0$. It remains to show that $(g_{r_1} \x g_{r_2})G=G$ for an arbitrary $r_1,r_2\in\R$. Let $(x,y)\in G$. Then, for any continuous function $\Psi : T^1M\x T^1M \to\R$, we have $\Psi((g_s \x g_s)(g_{r_1}x, g_{r_2}y)) = \Psi \circ (g_{r_1} \x g_{r_2}) (g_s x, g_s y)$ for all $s$. Since $(x,y) \in G$ and $\Psi \circ (g_{r_1} \x g_{r_2})$ is continuous, we see that 
 
	$$\lim\limits_{t\to\infty}\frac{1}{t}\int_{0}^{t}\Psi((g_s \x g_s)(g_{r_1}x, g_{r_2}y))\,ds = \int \Psi\circ(g_{r_1}\x g_{r_2})\,d(\mu\x\mu), $$
and by invariance of $\mu$, $\int \Psi\circ(g_{r_1}\x g_{r_2})\,d(\mu\x\mu)= \int\Psi\,d(\mu\x\mu).$ Hence, $(g_{r_1}x, g_{r_2}y) \in G$. Thus, $(g_{r_1} \x g_{r_2})G=G$ for all $r_1, r_2\in\R$.

	Since $\nu$ and $\mu\x\mu$ are equilibrium states, so is the measure $\frac{1}{2} \nu + \frac{1}{2}(\mu \x \mu)$. This measure is product expansive by Proposition \ref{prodexp}, so by Proposition \ref{Approximation} we can find  $U_t\subset \s{A}_t$ such that $\frac{1}{2}(\nu + (\mu \x \mu))(U_t\vartriangle P) \to 0$. In particular, we have $\nu(U_t) \to 1$ and $(\mu\x\mu)(U_t) \to 0$.

 We are now in a position to follow the proof of \S4.7 of \cite{CT}.  We know that $\mu \times \mu$ has the Gibbs property on $\tilde\GGG^M$. 
Now, taking $U_t$ as above and assuming without loss of generality that $\inf \nu(U_t) > 0$, by Lemma \ref{CombinedLemma}, for all sufficiently large $t$, we have that
$$\Lambda_t(\CCC\cap \tilde{\GGG}^M, \Phi, \gamma) \geq Ce^{tP(\Phi)}$$
for some $C$, where $\CCC = \{((x,y),t) \mid  x\in E_t\cap U_t\}$. Consequently, we have that
$$\sum_{(x,y)\in (\CCC\cap \tilde{\GGG}^M)_t} e^{\int_{0}^{t}\Phi(g_sx,g_sy)\,ds} \geq \Lambda_t(\CCC\cap\tilde{\GGG}^M, \Phi, \gamma)\geq  Ce^{tP(\Phi)}.$$
Observe that $B_t((x,y),\gamma/2)\subset U_t$ for all $(x,y)\in E_t\cap U_t$ because $\s{A}_t$ is adapted to $E_t$. Therefore, appealing to the Gibbs property shown in Lemma \ref{Gibbs}, we have that
$$(\mu\x\mu)(U_t)\geq \sum_{(x,y)\in (\CCC\cap \tilde{\GGG}^M)_t}Q_M^2e^{-tP(\Phi) + \int_{0}^{t}\Phi(g_sx,g_sy)\,ds}\geq Q_M^2C > 0.$$
However, this contradicts the fact that $(\mu\x\mu)(U_t)\to 0$. Thus, $\nu\ll (\mu \x \mu)$. This completes our proof that $\mu\x\mu$ has no mutually singular equilibrium states. We already showed that $\mu\x\mu$ is ergodic. We conclude that $\mu \x \mu$ is the unique equilibrium state for $(T^1M\x T^1M,(g_t\x g_t), \Phi)$. Applying Theorem \ref{Ledr}, we conclude that $\mu$ has the $K$-property.

\subsection{A general statement on obtaining the $K$-property} \label{s.general}
The argument described in this paper is rather flexible, and will apply for systems other than rank one geodesic flow. We state formally the abstract statement that is immediately provided by the proof above.  
\begin{thm} \label{general}
	Let $(X,\FFF)$ be a continuous entropy expansive flow on a compact metric space, and $\ph : X\to [0, \infty)$ a continuous potential. Suppose that every equilibrium measure for $\ph$ is almost expansive and that every equilibrium measure for $\Phi$ is product expansive for $(X \x X, \FFF \x \FFF)$.  Suppose that $X\x [0, \infty)$ has a $\lambda$-decomposition $(\PPP,\GGG,\SSS)$ with the following properties:
	\begin{enumerate}
		\item $\GGG$ has specification at all scales;
		\item $\ph$ has the Bowen property on $\GGG$;
		\item $P(\PPP \cup\SSS,\ph) < P(\ph)$,
	\end{enumerate}
	and that the corresponding  $\tilde \lambda$-decomposition $(\tPPP,\tGGG,\tSSS)$ for  $(X \x X, \FFF \x \FFF)$ satisfies
	\begin{enumerate}\setcounter{enumi}{3}
	\item $P(\tPPP \cup \tSSS,\Phi) < P(\Phi) = 2 P(\ph)$.
	\end{enumerate}
	Then   $(X \x X,\FFF \x \FFF,\Phi)$ has a unique equilibrium state, and thus the unique equilibrium state for $(X,\FFF,\ph)$ has the Kolmogorov property.
\end{thm}

We also note that our argument in \S \ref{lambdaestimate} shows that the pressure hypotheses (3) and (4) in the general result above hold if we can verify that
\[
\sup \{ P_\mu(\ph) : \int  \lambda \,d \mu=0\} < P(\ph) \text{ and } 
\sup \{ P_\nu(\Phi) : \int \tilde \lambda \,d \nu=0\} < P(\Phi) = 2 P(\ph).
\]

There is room for improvement in the hypotheses of Theorem \ref{general}. For example, the expansivity conditions could be replaced with a condition on the `pressure of obstructions to product expansivity' in the same spirit as \cite{CT}. Also, we expect that some of the hypotheses stated above can be shown (with more work) to be redundant; for example, one would like to argue that the pressure estimates (3) and (4) can be combined in general. We do not pursue these arguments here since they may distract from the main ideas needed for our approach. We expect to address an `optimal' general statement and explore further applications in future work.

We remark that Theorem \ref{general} applies in the case that $(X, \FFF)$ is a topologically mixing hyperbolic flow and $\varphi$ is a H\"older potential by taking $\GGG$ to be every orbit segment, and $\PPP, \SSS$ to be trivial. Specification and the Bowen property are satisfied globally, see \cite[\S7.3]{FH} for a convenient reference. Expansivity of the flow can easily be seen to guarantee the hypotheses on almost expansivity and product expansivity. We emphasize that classical hypotheses for uniqueness of equilibrium states for flows, e.g. those in \cite{eF77}, do not apply directly to $(X \x X,\FFF \x \FFF)$ since this product flow is not expansive.

\section{Bernoullicity of the Knieper-Bowen-Margulis measure} \label{s.bernoulli}
We recall results from the literature which allow us to conclude that  the $K$-property implies the Bernoulli property for the Knieper-Bowen-Margulis measure of maximal entropy $\mukbm$. Thus, we obtain the Bernoulli property for $\mukbm$. The argument for moving from $K$ to Bernoulli relies heavily on the foliation structure coming from non-uniform hyperbolicity of the system. Thus, this argument does not retain the level of generality of our arguments for the $K$-property.

\subsection{From K to Bernoulli} In the classic argument for the Bernoulli property by Ornstein and Weiss  \cite{OW73}, they first show the $K$-property. Then they argue that in their setting, which in \cite{OW73} was the geodesic flow on a constant negative curvature surface, the $K$-property implies the existence of a refining sequence of Very Weak Bernoulli partitions, which in turn implies Bernoulli. This argument was extended to equilibrium states for Anosov flows by Ratner \cite{mR74}. This has become the primary approach to proving the Bernoulli property in smooth dynamics, and was generalized by Pesin to non-uniformly hyperbolic flows \cite{yP77}. This strategy was also carried out by Chernov and Haskell for suspension flows over some non-uniformly hyperbolic maps with singularities \cite{CH96}, by Ledrappier, Lima, and Sarig for 3-dimensional smooth flows using countable state symbolic dynamics \cite{LLS}, and by Baladi and Demers for certain billiard flows \cite{BalDe}. We follow the account of \cite{CH96}, since it particularly emphasizes the details necessary for the flow case. We note that the arguments we need also appear in recent work by \cite{PTV} and elsewhere.

Most results in the literature are stated for a smooth measure or SRB measure, however it is widely accepted that what is really needed is product structure for the measure on rectangles. This is made clear in the account by Chernov and Haskell \cite{CH96}. Their results are stated for suspension flows over a non-uniformly hyperbolic map with a smooth measure, but their argument applies more generally. We explain how to extract a much more general statement from their write-up.

We claim that   `K implies Bernoulli' holds for a $C^2$ flow on a manifold equipped with a hyperbolic invariant measure $\mu$ if there exists an $\eps$-regular covering for $\mu$ for any $\eps>0$, where $\eps$-regular coverings are defined below. Section 5 of \cite{CH96} is devoted to showing that $\eps$-regular coverings exist for any $\eps>0$ when the measure $\mu$ is smooth. Section 6 of \cite{CH96} proves that if a measure $\mu$ is $K$ and there exists an $\eps$-regular covering with non-atomic conditionals for $\mu$ for any $\eps>0$, then any finite partition $\xi$ of the phase space with piecewise smooth boundary and a constant $C > 0$ such that $\mu(B(\partial \xi,\delta)) \leq C\delta$ for all $\delta > 0$ is Very Weak Bernoulli. A refining sequence of such partitions with diameter going to $0$ suffices to conclude the Bernoulli property for $\mu$. Such a sequence of partitions exist in this setting by \cite[Lemma 4.1]{OW98}. Thus, to conclude that $\mukbm$ is Bernoulli, we only need to show that $\eps$-regular coverings for $\mukbm$ exist for all $\eps>0$.

We recall that a \emph{rectangle} $R$ is a measurable set (which we can equip with a distinguished point $z \in R$) such that for all $x,y \in R$ the local weak stable manifold $W_x^{0s}$ and the local unstable $W^{u}_y$ intersect in a single point which lies in $R$. A rectangle $R\ni z$ is identified as the Cartesian product of $W^{u}_z\cap R$ with $W^{0s}_z\cap R$, and there is a natural product measure $\mu_R^p = \mu_z^{u}\x\mu_z^{0s}$, where $\mu_z^{u}$ is the conditional measure induced by $\mu$ on $W^{u}_z\cap R$, and $\mu_z^{0s}$ is the corresponding factor measure on $W^{0s}_z$. We give Chernov and Haskell's definition of $\eps$-regular covering here.
\begin{defn} \label{regcover}
Given any $\eps > 0$, we define an \emph{$\eps$-regular covering for $\mu$} of the phase space $M$ to be a finite collection of disjoint rectangles $\RRR = \RRR_\eps$ such that
\begin{enumerate}
	\item $\mu(\bigcup_{R\in\RRR} R) > 1 - \eps$
	\item Given any two points $x,y\in R\in\RRR$, which lie in the same unstable or weakly stable manifold, there is a smooth curve on that manifold which connects $x$ and $y$ and has length less than $100\cdot\operatorname{diam} R$
	\item For every $R\in\RRR$, with distinguished point $z \in R$, the product measure $\mu_R^p = \mu_z^{u}\x\mu_z^{0s}$ satisfies $\abs{\mu_R^p(R)/\mu(R) - 1} < \eps$. Moreover, $R$ contains a subset $G$ with $\mu(G) > (1-\eps)\mu(R)$ such that for all $x \in G$, $\abs{(d \mu_R^p/ d\mu)(x) - 1} < \eps.$
\end{enumerate}
\end{defn}

\subsection{Constructing an $\eps$-regular covering for $\mukbm$}
The measure $\mukbm$ is hyperbolic because  $\mukbm(\Reg)=1$, see \cite[Corollary 3.7]{BCFT}. Let $\eps > 0$. By \cite[Lemma 8.3]{yP77} and \cite[Lemma 1.8]{Pes77}, for any hyperbolic measure $\mu$, we can find a finite collection of disjoint rectangles $R$ covering a Pesin set for the Lyapunov regular points for $\mu$. Applying this to $\mukbm$ and a choice of Pesin set with measure at least $1-\eps$ gives the first condition.  The rectangles $R$ can be chosen with maximum diameter as small as we like. The second condition is immediate since the leaf metrics are uniformly equivalent to the Riemannian distance on small leaves. 

This leaves only the third condition to check for rectangles with sufficiently small diameter. We recall Knieper's construction of $\mukbm$ from \cite{knieper98} which gives us the product structure we need. Writing $\tilde{M}$ for the universal cover of $M$ and $\tilde{M}(\infty)$ for the boundary at infinity, let $\nu_p$ be a non-atomic measure on $\tilde{M}(\infty)$ described in \cite{knieper98}. Let $\VVV$ be the set of $(\xi,\eta) \in \tilde{M}(\infty) \x \tilde M(\infty)$ such that there exists a geodesic $\gamma$ with $\gamma(-\infty) = \xi, \gamma(\infty) = \eta$,  and let 
$$
P^{-1}(\xi,\eta) = \{\text{geodesics }\gamma\mid \gamma(-\infty) = \xi,\,\gamma(\infty) = \eta\}.
$$
Knieper defines the measure $\mukbm$ by setting for $A\subset T^1\tilde{M}$,
$$\mukbm(A) = \int_{\VVV}\operatorname{Vol}(\pi(P^{-1}(\xi,\eta)\cap A))f(\xi,\eta)\,d\nu_p(\xi)\,d\nu_p(\eta)$$
where $f(\xi,\eta) = e^{-h(b_p(q,\xi) + b_p(q,\eta))}$ with $q$ any point on the geodesic connecting $\xi$ and $\eta$, $p\in \tilde{M}$, and $b_p(q,\xi)$ is a Busemann function. The definition of $f$ is independent of the choice of $q$, and we know that $f$ is continuous by \cite[Chapter II]{wB95}. The measure is shown to be equivariant under the fundamental group, and thus descends to a measure on $T^1M$. Knieper shows that this characterization of $\mukbm$ defines the unique measure of maximal entropy.

By the flat strip theorem, if $P^{-1}(\xi,\eta)$ contains a regular geodesic, then this is the only geodesic in $P^{-1}(\xi,\eta)$, which we can write as $\gamma(\xi, \eta)$. Because $\mukbm(\Reg) = 1$, then the integrand $\operatorname{Vol}(\pi(P^{-1}(\xi,\eta)\cap A))$ just becomes the Lebesgue measure along $\gamma(\xi, \eta)$ of the set $A$ for $(\nu_p \x \nu_p)$-almost every $(\xi, \eta) \in \VVV$. We see that $d\mukbm=  f(\xi,\eta)d\nu_p\x d\nu_p \x dt$. In the terminology of \cite{mB02},  the measure $f(\xi,\eta)d\nu_p\x d\nu_p$ is a geodesic current with the quasi-product property. It follows that $\mukbm$ is a product measure on the unstable and weak stable manifolds, because there is a natural identification of stable and unstable manifolds of $v$ with subsets of $\tilde{M}(\infty)$.

Now we will show that our rectangles satisfy condition (3) in the definition of an $\eps$-regular cover. Let $R$ be a rectangle of sufficiently small diameter. Since stable and unstable manifolds at $v$ intersect transversally if and only $v \in \Reg$, it follows that if a rectangle $R$ is well-defined, then $R \subset \Reg$. For $z\in R$, let $(\xi_z,\eta_z)$ be the corresponding element of $\VVV$. The conditional measure $\mu_z^u$ on $R\cap W^u(z)$ is given by $d\mu_z^u(\eta) = \frac{f(\xi_z,\eta)\,d\nu_p(\eta)}{\int_{W^u(z)\cap R}f(\xi_z,\eta)\,d\nu_p(\eta)}.$ Since $f$ is continuous, by taking $R$ with a small enough diameter, we have that 
$\abs{d\mu_z^u / d\mu_w^u - 1} \leq \eps$ for $z, w\in R$.

This is sufficient to show condition (3) of an $\eps$-regular covering by integrating this derivative and appealing to the definition of conditional measures. This shows the existence of an $\eps$-regular covering. We conclude that $\mukbm$ is Bernoulli.

\subsection{The power of product structure}  \label{pp}
Product structure for measures is an extremely powerful tool in ergodic theory. The product structure described above is what Babillot used to obtain mixing for $\mukbm$ \cite{mB02}.  For non-uniformly hyperbolic maps with a smooth measure, it is shown by Pesin-Katok-Strelcyn theory \cite{KS86, yP77} that an ergodic component decomposes into a finite union of subcomponents of equal measure which are cyclically permuted by the map, and the corresponding iterate of the map is $K$ on each component. Thus, mixing implies Bernoulli in that setting. As noted in \cite{CH96}, it is widely believed that flow versions of this statement hold. It is also expected that this part of the theory goes through with a product structure assumption on a hyperbolic measure in place of a smoothness assumption. The paper \cite{OW98}, while focused on the SRB measure, makes this strategy clear. However, that paper does not contain convenient statements to reference, and is more focused on the big picture rather than full details, particularly in the flow case.  Since mixing for $\mukbm$ was proved by Babillot, there are no rotation factors, so it is likely that the approach discussed above would give the Bernoulli property for $\mukbm$ without the need for the novel arguments for the $K$-property which are presented in this paper. We stress that the necessary details (or even precise statements) of this approach are not written, and a full account will require elucidating a number of non-trivial technical details. We hope this will be rectified in the future. The current paper is to the best of our knowledge the first time that the Bernoulli property, or even the $K$-property, for $\mukbm$ has been claimed in the literature.

 We note that the symbolic dynamics recently obtained by Araujo-Lima-Poletti \cite{ALP} provides product structure at the symbolic level for the equilibrium states considered in this paper. Product structure for the symbolic lifts is sufficient to improve $K$ to Bernoulli. This argument is detailed in \cite[Corollary 1.3]{ALP}. We note that their symbolic construction is highly involved, and was not available in higher dimensions when our preprint first appeared. A geometric construction of product structure for the equilibrium states considered in this paper, for example extending the Knieper construction, is not currently known.

A methodological advantage of our approach to the $K$-property is that we do not use arguments based on product structure. We expect this will be an advantage of the approach developed here in settings where obtaining product structure is difficult or does not make sense.

\subsection*{Acknowledgments} We would like to thank Omri Sarig, Fran\c{c}ois Ledrappier, Yves Coud\`ene and Ali Tahzibi for helpful conversations. We would also like to thank the anonymous referee for many helpful comments.

	\bibliographystyle{amsplain}
	\bibliography{Kpropertyreferences}

\end{document}